\documentclass[12pt]{amsart}
\usepackage[utf8]{inputenc}
\usepackage{graphicx,amssymb,amsmath,amsthm,bm}
\usepackage{caption}
\usepackage{subcaption}
\usepackage{indentfirst}
\usepackage{cancel}
\usepackage{lipsum}
\usepackage{epstopdf,epsfig}
\usepackage{enumerate,color} 
\usepackage{comment}
\usepackage[color=yellow]{todonotes}
\usepackage{kotex}

\setuptodonotes{inline}
 
\usepackage{hyperref}

\usepackage{array,ragged2e}
\newcolumntype{C}{>{\Centering\arraybackslash}m{0.14\linewidth}}

\usepackage[color=yellow]{todonotes}
\setuptodonotes{inline}

\usepackage[left=1.5cm, right=1.5cm, top=2cm, bottom=2.5cm, papersize={16.5cm,25.2cm}]{geometry} 
\usepackage{accents}

\usepackage[overload]{empheq}
\usepackage{soul}
\usepackage[multiple]{footmisc}
\numberwithin{equation}{section}
\everymath{\displaystyle}

\theoremstyle{plain}
\newtheorem{theorem}{Theorem}[section]
\newtheorem{lemma}[theorem]{Lemma}
\newtheorem{proposition}[theorem]{Proposition}

\theoremstyle{definition}

\newtheorem*{defi*}{Definition} 
\theoremstyle{remark}
\newtheorem{remark}{Remark}

\theoremstyle{remark}

\usepackage{etoolbox}
\makeatletter
\let\alignts@preamble\align@preamble
\patchcmd{\alignts@preamble}{\displaystyle}{\textstyle}{}{}
\patchcmd{\alignts@preamble}{\displaystyle}{\textstyle}{}{}

\def\alignts{\let\align@preamble\alignts@preamble\start@align\@ne\st@rredfalse\m@ne}

\makeatother

\allowdisplaybreaks

\title[Traveling waves]{Quadratic Growth Model with Discontinuity: A Link between Monostable and Bistable Traveling Waves}

\author[W. Choi]{Wonhyung Choi}
\address[WC]{School of Computer Engineering and Applied Mathematics, Hankyong National University, 327, Jungang-ro, Anseong, 17579, Republic of Korea}
\email{whchoi@hknu.ac.kr}

\author[J. Bae]{Junsik Bae*}
\address[JB]{IRMAR - UMR 6625, CNRS,  Univ Rennes, Rennes, 35000, France}
\email{altena00@gmail.com}

\author[Y.-J. Kim]{Yong-Jung Kim}
\address[YK]{Department of Mathematical Sciences, KAIST, 291, Daehak-ro, Yuseong-gu, Daejeon, 34141, Republic of Korea}
\email{yongkim@kaist.edu}

\date{\today} 


\begin{document}
\begin{abstract}
We classify traveling waves and stationary solutions of a reaction–diffusion equation arising in population dynamics with Allee-type effects. The reaction term is given by a quadratic polynomial with a discontinuity at zero, which captures finite-time extinction for sub-threshold populations. This discontinuity induces a free boundary in the wave profile, a phenomenon that distinguishes the model from the classical logistic or Allen–Cahn equations. A complete scenario is presented that connects monostable and bistable traveling waves through the wave speed parameter, thereby providing a unified framework for their dynamics.

\end{abstract}

\maketitle

\section{Introduction}

This paper has two primary objectives. The first is to introduce the quadratic reaction function
\begin{equation}\label{P2}
\dot{u} = r_0 + r_1 u + r_2 u^2, 
\qquad r_0 \le 0,\ r_2 < 0,\ r_1 > \sqrt{4r_0r_2},
\end{equation}
from the perspective of population dynamics. Here, \(u \ge 0\) denotes the population density, and \(\dot{u}\) represents its temporal growth rate. In this model, \(r_1>0\) corresponds to the intrinsic growth rate, while \(r_2<0\) reflects the reduction caused by intraspecific competition: the linear term \(r_1 u\) models population increase, whereas the quadratic term \(r_2 u^2\) captures competitive suppression. The constant term \(r_0\) is typically interpreted as a \emph{harvesting term} when \(r_0 < 0\) and as an \emph{external source term} when \(r_0 > 0\). While the terms in \eqref{P2} admit the aforementioned mechanistic interpretations, we emphasize that the model \eqref{P2} should be understood as a second-order polynomial approximation of population dynamics, without necessarily assigning a biological meaning to each term.

The second objective is to classify traveling wave solutions of the associated reaction–diffusion equation and to present a comprehensive scenario that connects monostable and bistable traveling waves by treating the wave speed as a parameter.

\subsection{Full second-order polynomial model} \label{sec:model}

A well-known special case of the quadratic model \eqref{P2} is the logistic equation, which corresponds to the case $r_0=0$ and is written as
\begin{equation}\label{Logistic}
\dot{u} = r_1 u + r_2 u^2 = r u(\beta - u), \quad r = -r_2 > 0,\ \beta = -\frac{r_1}{r_2} > 0.
\end{equation}
Here, $\beta > 0$ is known as the carrying capacity. One notable feature of the logistic equation \eqref{Logistic} is the \emph{hair-trigger effect}, which means that the solution always converges to the carrying capacity $\beta$ as long as the initial value is positive. This occurs because $u = 0$ is an unstable equilibrium, while $u = \beta$ is a stable one. Species exhibiting the hair-trigger effect are appropriately modeled using the logistic model. 

In contrast, some species exhibit the \emph{Allee effect} \cite{Al,KDLD}, where the population becomes extinct if its density falls below a certain threshold, say $\alpha > 0$. To model such behavior, the Allen–Cahn type model is often employed:
\begin{equation}\label{Allen-Cahn}
    \dot{u} = r u (u - \alpha)(\beta - u), \quad 0 < \alpha < \beta,\ r > 0,
\end{equation}
where \(u = \alpha\) is an unstable equilibrium, while \(u = 0\) and \(u = \beta\) are stable equilibria. Expanding the right-hand side gives
\[
\dot{u} = r_1 u + r_2 u^2 + r_3 u^3,
\]
which is a cubic polynomial model without a constant term. The signs of the coefficients are
\[
r_1 = -r \alpha \beta < 0, \quad r_2 = r(\alpha + \beta) > 0, \quad r_3 = -r < 0.
\]
It is noteworthy that the coefficient of the linear term is negative, while that of the quadratic term is positive, indicating that this cubic polynomial model is fundamentally different from the logistic population model.

The quadratic growth model \eqref{P2} preserves the logistic structure while accommodating the Allee effect. When $r_0 < 0$, \eqref{P2} has two distinct positive equilibria due to the condition $r_1 > \sqrt{4r_0r_2}$. Denoting these equilibria as $\alpha$ and $\beta$ with $0 < \alpha < \beta$, the model is rewritten as
\[
\dot{u} = r(u - \alpha)(\beta - u).
\]
We may set $r=1$ and $\beta = 1$ through appropriate choice of the population and time units. Moreover, since ecological models assume the population process ceases upon extinction ($u = 0$), we only consider the nonnegative domain $u \ge 0$. To make this explicit, we formulate the model as follows:
\begin{equation}\label{Model}
\dot{u} = (u - \alpha)(1 - u)\, \chi_{\{u > 0\}},\qquad 0<\alpha<1,
\end{equation}
where the characteristic function $\chi_{\{u > 0\}}$ equals 1 when $u > 0$ and 0 otherwise, ensuring the process halts once extinction occurs. On the right-hand side, the discontinuity occurs at \(u=0\), which reflects the fundamental fact that once the population vanishes, the population dynamics cease.

The resulting nonlinearity in \eqref{Model} exhibits bistability: the values \( u = 0 \) and \( u = 1 \) are stable equillibriums, while \( u = \alpha \) is an unstable one. The key distinction from the Allen-Cahn type nonlinearity in \eqref{Allen-Cahn} lies in the discontinuity at \( u = 0 \), in addition to the opposite signs of the first and second order terms. One may consider the Allen-Cahn type nonlinearity in \eqref{Allen-Cahn} as the regularization of \eqref{Model} by multiplying the reaction term by $u$. A key consequence of the discontinuity in \eqref{Model} is that if the initial population is less than \( \alpha \), the solution undergoes extinction at a finite time $T>0$ and remains \( u(t) = 0 \) for all \( t \ge T \). This finite-time extinction introduces a free boundary into the traveling wave solution, enriching its structure in comparison with the standard logistic or Allen–Cahn models. Understanding this phenomenon and characterizing the corresponding traveling wave solutions form the main theme of this paper.

\subsection{Traveling wave solutions of a reaction-diffusion equation}

We now consider a reaction-diffusion equation by adding a diffusion term to the quadratic population dynamics model:
\begin{equation}\label{sys}
u_t = u_{xx} + (u - \alpha)(1 - u)\, \chi_{\{u > 0\}}, \quad 0 < \alpha < 1,\quad t>0,\quad x \in \mathbb{R}.
\end{equation}
Following the standard convention, we set the diffusion coefficient as $1$, which can be achieved via appropriate scaling in the spatial variable $x$. We investigate the existence of traveling wave solutions to \eqref{sys} and their structure. 

A traveling wave solution with speed $c \in \mathbb{R}$ is a solution of the form $u(x,t) = u(x - ct)$, and under the change of variable $\xi = x - ct$, the equation becomes
\begin{equation}\label{tw}
-c u'(\xi) = u''(\xi) + (u(\xi) - \alpha)(1 - u(\xi))\, \chi_{\{u > 0\}}, \quad \xi \in \mathbb{R},
\end{equation}
where $'$ denotes differentiation with respect to $\xi$. Note that if $u(\xi)$ is a solution of \eqref{tw}, so is $u(\xi + \xi_0)$ for any constant $\xi_0 \in \mathbb{R}$. Thus, the uniqueness of a traveling wave solution is always understood up to translation. We also note that \eqref{tw} has a symmetry
\begin{equation}\label{Symmetry}
c \mapsto -c \quad \text{and} \quad \xi \mapsto -\xi.    
\end{equation}
In this work, we primarily focus on classical solutions—those that are continuously differentiable on $\mathbb{R}$ and a discontinuity in the second derivative may occur at the interface, where $u=0$, due to a discontinuity in the reaction term. Furthermore, we restrict our attention to nonnegative bounded solutions.

\begin{remark}
The regularity $u'(\xi)\in \textrm{Lip}(\mathbb{R})$ ensures the validity of the comparison principle, which in turn guarantees the uniqueness of weak solutions to the initial value problem associated with \eqref{sys}, \cite{CKKP}. See also the Appendix of the present paper. In Section~\ref{Sec3_1}, we use the constructed traveling waves, along with the comparison principle, to analyse the dynamics of solutions to the initial value problem.
\end{remark}
We classify nonnegative traveling wave solutions to \eqref{tw} into three types according to the equillibriums they connect. The first type connects the states \( u = 1 \) and \( u = \alpha \), satisfying
\begin{equation}\label{a2}
    \lim_{\xi \to -\infty} u(\xi) = 1, \quad \lim_{\xi \to +\infty} u(\xi) = \alpha.
\end{equation}
This case corresponds to a \emph{Fisher-KPP type} traveling wave, which is a monostable wave. The second type connects the states \( u = 1 \) and \( u = 0 \), satisfying
\begin{equation}\label{a1}
    \lim_{\xi \to -\infty} u(\xi) = 1, \quad u(\xi) = 0 \text{ if and only if } \xi \ge 0.
\end{equation}
This is a \emph{bistable} traveling wave. We will see that the second type is obtained as the limit case of the first type.

There exists a minimal wave speed for monotone traveling wave solutions satisfying \eqref{a2}, which is given as
\[
c^*:= 2\sqrt{1 - \alpha}.
\]
For all speeds $c \ge c^\ast$, there exists a unique, positive, monotone traveling wave solution satisfying~\eqref{a2}. However, if $c < c^\ast$, the existence and structure of solutions differ between the two cases $\alpha < 1/3$ and $\alpha \ge 1/3$. Below, these two cases are separated and organized into two theorems.

\begin{figure}[ht]
    \centering
  \includegraphics[width=\linewidth]{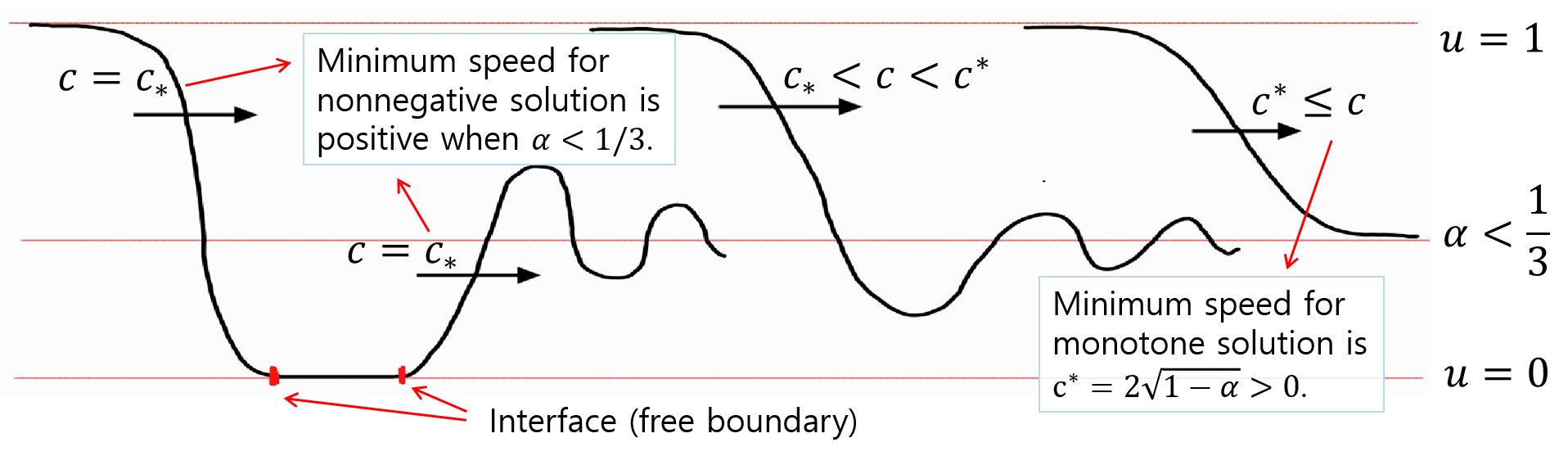}
  \caption{The solution profiles in Theorem~\ref{thm1} illustrate the link between monostable and bistable traveling waves for the case \(0<\alpha<1/3\).
}
  \label{NewFig1}
\end{figure}

\begin{theorem}\label{thm1}
Consider \eqref{tw} with $0 < \alpha < 1/3$.\\
(1). The solution of \eqref{tw} with the monostable boundary condition \eqref{a2} exists if and only if $c\geq c_*$ for some $c_*\in(0,c^*)$, where $c_*$ depends on $\alpha$. Furthermore,
\begin{enumerate}[(i).]
\item If $c \ge c^*$, the solution is unique, strictly positive, and monotone. 
\item If $c_*<c<c^*$, the solution is unique, strictly positive, and oscillates toward $\alpha$ as $\xi \to +\infty$.
\item If $c=c_*$, the solution is not unique; it either vanishes at a point or on an interval, and oscillates toward $\alpha$ as $\xi \to +\infty$. 
\end{enumerate}
(2). The solution of \eqref{tw} with the bistable boundary condition \eqref{a1} exists if and only if $c=c_*$. The solution is unique and monotone.
\end{theorem}

Statements analogous to $(i)$ and $(ii)$ of this theorem also hold for the Allen–Cahn type reaction case. The key difference arises in $(iii)$: here, non-uniqueness is caused by the discontinuity of the reaction function at $u = 0$, which violates the Lipschitz continuity assumption that ensures uniqueness in ODE systems. If $c > c_*$, the solution remains positive. If $c = c_*$, the solution touches zero and non-uniqueness occurs. For instance, one such solution increases immediately as $\xi$ increases and then oscillates around $u = \alpha$. Another may remain at zero over some interval before increasing, as illustrated in Figure~\ref{NewFig1}. All of these are traveling wave solutions of \eqref{tw} connecting $u = 1$ and $u = \alpha$.

If the solution remains at zero beyond the interface, then it does not satisfy the boundary condition~\eqref{a2}. Instead, it is the unique solution satisfying the bistable boundary condition~\eqref{a1}, as stated in the second part of Theorem~\ref{thm1}. It is well known that, for the Allen–Cahn type reaction, the traveling wave speed connecting two stable equilibria is unique. The theorem shows that this property also holds for the discontinuous reaction case considered in this paper. This bistable traveling wave is expected to be the stable one that emerges from general initial data (see Figure~\ref{Fig1_1}). The positivity of its speed, $c_* > 0$, indicates propagation of the species. Hence, we conclude that if $0 < \alpha < 1/3$, the species overcomes the Allee effect and expands over the space.

\begin{figure}[ht]
    \centering
   \includegraphics[width=\linewidth]{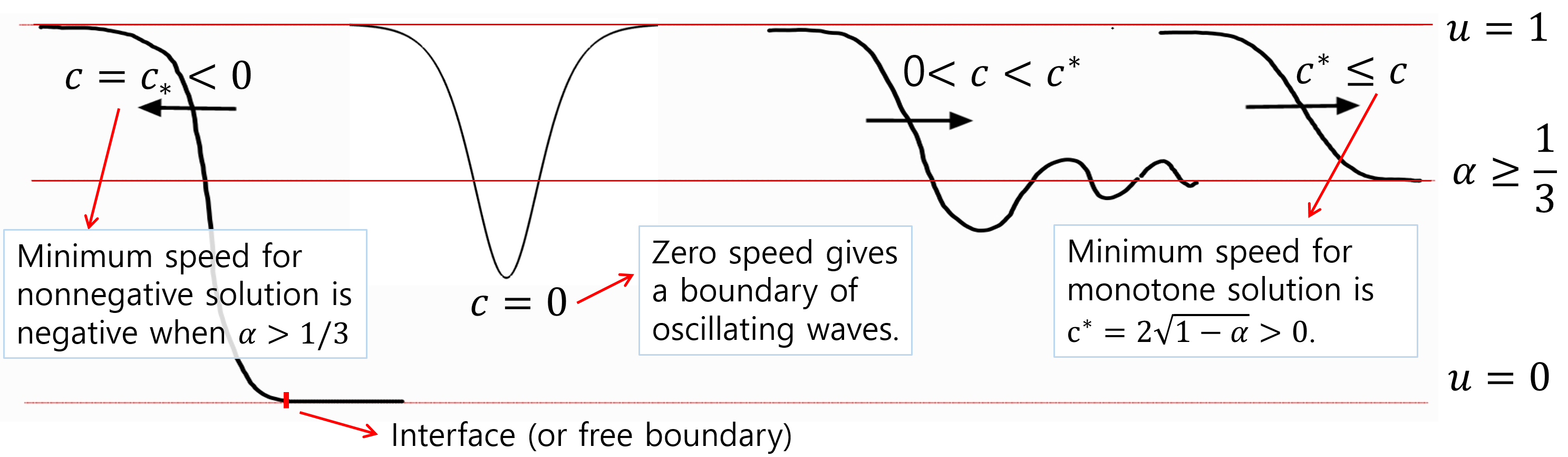}
  \caption{The solution profiles in Theorem~\ref{thm2} illustrate the link between monostable and bistable traveling waves for the case $1/3\le\alpha<1$.   }
  \label{NewFig2}
\end{figure}

Next, we consider the case $1/3\leq \alpha<1$. 
\begin{theorem}\label{thm2}
Consider \eqref{tw} with  $1/3\le \alpha<1$.\\ 
(1). The solution of \eqref{tw} with the monostable boundary condition \eqref{a2} exists if and only if $c>0$. Furthermore,
\begin{enumerate}[(i).]
\item If $c \ge c^*$, the solution is unique, strictly positive, and monotone. 
\item If $0<c<c^*$, the solution is unique, strictly positive, and oscillates toward $\alpha$ as $\xi \to +\infty$.
\end{enumerate}
(2). For each $\alpha\in[1/3,1)$, there exists a unique speed $c_*\leq 0$ such that the solution of \eqref{tw} with the bistable boundary condition \eqref{a1} exists if and only if $c=c_*$. The solution is unique and monotone. We have $c_\ast =0$ for $\alpha=1/3$, and $c_\ast <0$ for $1/3<\alpha<1$.
\end{theorem}

An important distinction in the case \( \alpha > 1/3 \) is that the bistable wave propagates with a \emph{negative} speed \( c_* < 0 \). Physically, this corresponds to \emph{species extinction}, while mathematically, it highlights the role of another critical wave speed, namely zero. When \( c = 0 \), the solution is symmetric, as illustrated in Figure~\ref{NewFig2}, and it connects the state \( u = 1 \) to itself. In contrast, when \( c < 0 \), the solution diverges beyond the three equilibria and no longer satisfies the boundary condition \eqref{a2}. Hence, traveling wave solutions of \eqref{tw} connecting distinct equilibria do not exist for $c<0$ with \(c\neq c_\ast\). Numerical simulations (see Figure~\ref{Fig3_1}) indicate that the traveling wave with speed $c_\ast<0$ is stable. In Section~\ref{Sec3_1}, based on this traveling wave, we show that a compactly supported initial population becomes extinct in finite time when $\alpha>1/3$.

Lastly, the third type of traveling wave solution connects the equilibrium states \(u = 0\) and \(u = \alpha\), subject to the boundary conditions
\begin{equation}\label{a4}
u(\xi) = 0 \quad \text{for } \xi \le 0, 
\qquad \lim_{\xi \to +\infty} u(\xi) = \alpha.
\end{equation}
Although this case also corresponds to a monostable traveling wave, it represents a \emph{pushed wave}, which is distinct from the classical Fisher–KPP type wave.

\begin{figure}[ht]
    \centering
    \includegraphics[width=\linewidth]{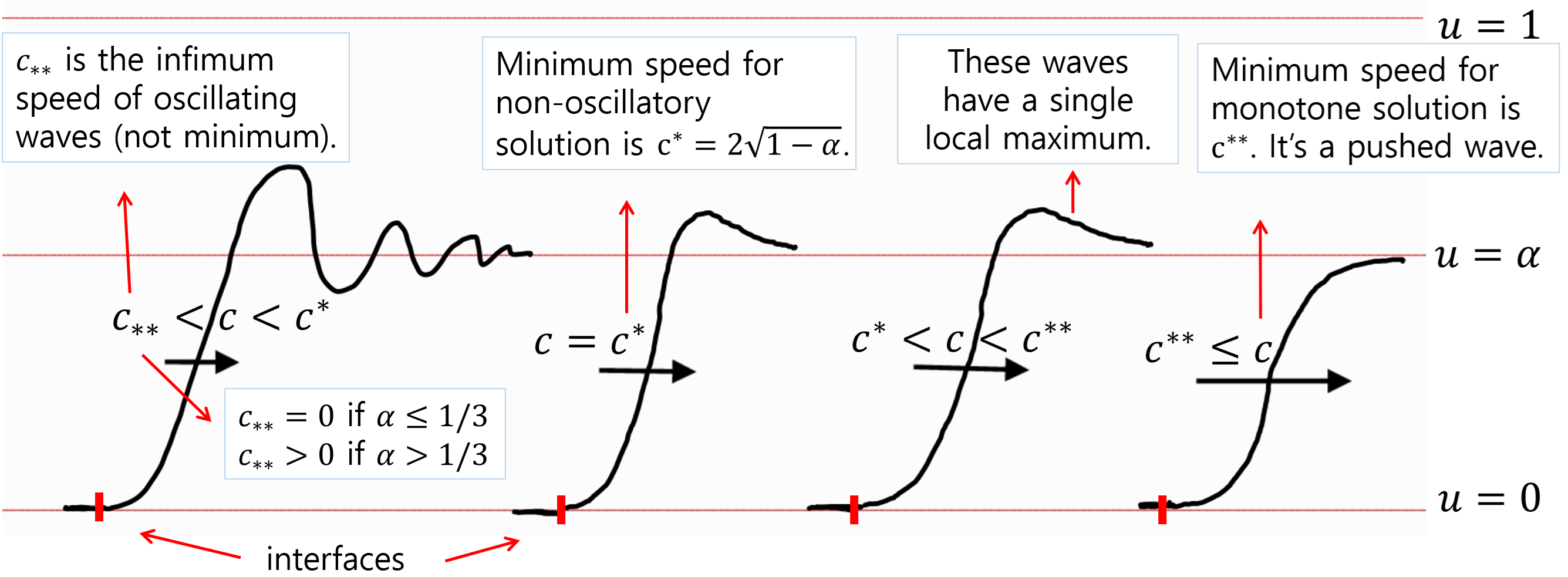}
  \caption{Pushed traveling waves introduced in Theorem~\ref{thm3}.}
  \label{NewFig3}
\end{figure}

\begin{theorem}\label{thm3} Consider \eqref{tw} with  $0<\alpha<1$ and the monostable boundary condition \eqref{a4}. \\
(1). There exists a constant $c_{**}\ge0$ such that the solution of \eqref{tw} with \eqref{a4} exists if and only if $c_{**}<c$. In particular, $c_{**}=0$ if $0<\alpha\le1/3$,  and $c_{\ast\ast}=-c_*>0$ if $\alpha>1/3$, where $c_*$ is the one in Theorem \ref{thm2}.\\
(2). There exists $c^{**}\in [\max(c_{**},c^*),2]$, where the followings hold:
\begin{enumerate}[(i).]
    \item If $c\ge c^{**}$, the solution is monotone.
    \item If $c\in[c^*,c^{**})$, the solution has a single local maximum.
    \item  If $c\in(c_{**},c^*)$, the solution oscillates toward $\alpha$ as $\xi \to +\infty$. 
\end{enumerate}
\end{theorem}

In Theorem~\ref{thm3}, it holds that \(c_{**}<c^{**}\) (see the proof of Theorem~\ref{thm3}). However, the precise relations involving the wave speed \(c^\ast\) have not yet been fully established. For example, if it holds that \(c^\ast \leq c_{\ast\ast}\), then no oscillatory solution exists. Similarly, if \(c^\ast = c^{\ast\ast}\), then there exists no solution with exactly one local maximum. The numerical plot in Figure~\ref{Fig4_1} reveals that the interval \((c^\ast, c^{**})\) is not empty. Clarifying these relations would complete the theoretical picture, and we suggest that further attention be directed to this open problem.
 
\begin{figure}[ht]
\centering
\includegraphics[width=\linewidth]{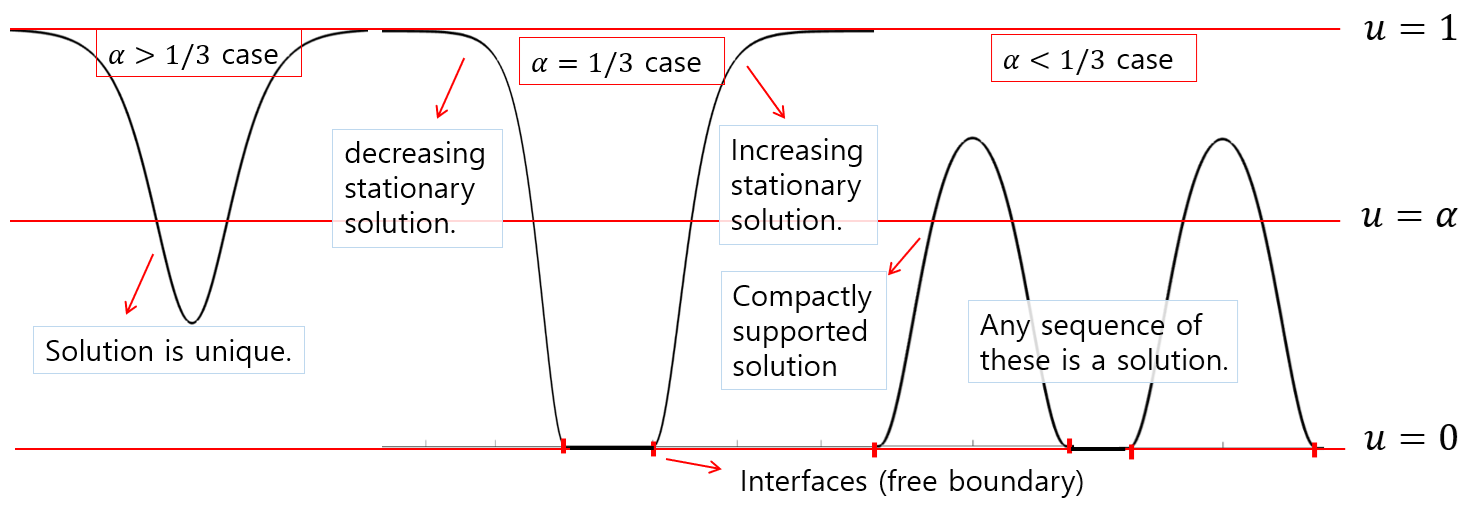}
  \caption{Stationary solutions introduced in Theorem~\ref{thm4}.}
\label{NewFig4}
\end{figure}

Stationary solutions are of particular importance. We summarize them in the following theorem.

\begin{theorem}\label{thm4}
Consider the solution of \eqref{tw} with $c=0$. \\
(1). If \( \alpha < 1/3 \), there exists a solution with a connected compact support and maximum bounded by \(\alpha < \max u(\xi) < 1\). Moreover, any sequence of such solutions also constitutes a solution.\\
(2). If $\alpha=1/3$, the bistable solution of Theorem~\ref{thm2} is a stationary solution to \eqref{tw}. Additional solutions can be constructed by gluing it with its reflection, separated by an interval of various sizes. \\
(3). If $\alpha>1/3$, the solution satisfying \(u(\pm\infty)=1\) is unique, and it holds that \(0<\min u(\xi)<\alpha\). 
\end{theorem}

We remark that the stationary solution with a single bump in Theorem~\ref{thm4}(1) for the case \(\alpha<1/3\) coincides with that studied in \cite{CKKP}, where the associated Dirichlet–Neumann problem was considered. Moreover, numerical simulations (see Figure 6 in \cite{CKKP}) suggest that this solution serves as a threshold, separating population expansion from finite-time extinction. Therefore, it is asymptotically unstable. In Section~\ref{Sec3_1}, we present an improved version of the finite-time extinction result (Theorem 6.2 in \cite{CKKP}) and establish the emergence of free boundaries: compactly supported initial data evolve with free boundaries. In contrast, the stationary solution with a single dip for the case \(\alpha>1/3\) represents the dual case, serving as a threshold that separates population shrinkage from recovery to \(u=1\). Hence, it is also asymptotically unstable. The stationary solution in the borderline case \(\alpha=1/3\) appears to be asymptotically stable according to the numerical results (see Figure~\ref{Fig8_1}).

Perhaps the most distinctive property of the traveling wave solutions introduced above is the emergence of an interface, or free boundary. When a traveling wave solution touches the value \(0\), a free boundary arises that separates the region of positive values from the zero region. Such free boundaries occur in a wide variety of problems. From the perspective of traveling waves for reaction–diffusion equations, free boundaries may arise. For instance, Hilhorst \emph{et al.}~\cite{Hilhorst2008} showed that in the Fisher–KPP equation with degenerate diffusion, the traveling wave solution with minimal speed possesses a free boundary, whereas all other traveling wave solutions do not. By contrast, in the classical Fisher–KPP equation with linear diffusion, no free boundary arises. Two mechanisms account for this: first, an intrinsic property of linear diffusion ensures that even when the support of the initial data is compact, the solution of the heat equation remains strictly positive for all $t>0$; second, certain traveling waves are generated not by diffusion but by synchronized growth induced by the reaction term (see~\cite{Hilhorst2016}), a mechanism that requires the presence of a tail. Consequently, except for the minimal-speed monotone traveling wave, none of the traveling waves of the Fisher–KPP equation gives rise to a free boundary, regardless of modifications to the diffusion operator.  

A system with linear diffusion that admits a free boundary can be found in models with Stefan-type boundary conditions imposed at the edge of the solution’s support. Such models have been the subject of extensive subsequent study, and many of their properties are now well understood (see~\cite{Du2010} and the references therein). However, the corresponding traveling wave solution is a weak solution of the PDE only inside its support, and fails to be a weak solution on the entire real line.  

In contrast to the monostable case, bistable traveling waves connecting two stable equilibria admit a unique wave speed, so no additional reaction-induced waves exist. Hence, one may predict that the Allen–Cahn reaction \eqref{Allen-Cahn}, when coupled with degenerate diffusion, produces free boundaries, as confirmed in the work of Y. Hosono~\cite{Ho}. The free boundary examined in the present paper, however, is of a different nature: neither degenerate diffusion nor artificially imposed Stefan-type boundary conditions are involved. Instead, it originates from the discontinuity of the reaction function near \(0\). See also \cite{Giletti2022}, where the so-called terrace solutions are studied.\\

The proofs of Theorems~\ref{thm1}--\ref{thm4} are based on phase plane analysis and a shooting argument. We first consider the equation \eqref{tw} \textit{without} the step function:
\begin{equation}\label{KPP_Trav}
-cu' = u''  + (u-\alpha)(1-u).
\end{equation} 
Using phase plane analysis, we carefully examine the behaviour of the trajectories, which depend on $c$. In particular, we look for the wave speed at which the trajectory of \eqref{KPP_Trav}, starting from the unstable manifold at $(u,u')=(1,0)$, reaches the point $(u,u')= (0,0)$ in finite $\xi=\xi_0$. We then glue this solution to the trivial solution $(u,u')\equiv (0,0)$ of \eqref{tw}, and as a result, we prove the existence of nontrivial traveling waves of \eqref{tw} with free boundaries. Combined with the standard argument, \cite{BS,PW}, the shooting argument allows us to classify all possible non-trivial traveling wave solutions of \eqref{tw} satisfying the far-field conditions, as well as its stationary solutions.

We remark that the work of \cite{CK} studies monotone solutions to \eqref{tw} satisfying \eqref{a1} by regularizing the reaction term and then taking the limit. However, our approach does not require such a limiting process.

The paper is organized as follows. In Section~\ref{Sec_2}, we prove Theorems~\ref{thm1}–\ref{thm4}. In Section~\ref{Sec_3}, we discuss the behavior of initial value problems associated with \eqref{sys}. In Section~\ref{Sec3_1}, we observe some key features of the model \eqref{sys}, including finite-time extinction and the emergence of free boundaries. In Section~\ref{Sec3_2}, we present numerical simulations supporting the stability of traveling wave and stationary solutions of \eqref{sys} with free boundaries, and address mathematical difficulties that arise in stability analysis.

\section{Proof of Main results}\label{Sec_2}
We rewrite the equation \eqref{KPP_Trav} as follows: 
\begin{equation}\label{KPP_Sys}
\left\{
\begin{array}{l l}
u' = w, \\
w' = -cw  - (u-\alpha)(1-u).
\end{array} 
\right.
\end{equation}
Due to symmetry \eqref{Symmetry}, it suffices to consider the case $c\geq 0$.

The system \eqref{KPP_Sys} has two equilibrium points $(u,w)=(\alpha,0)$ and $(u,w)=(1,0)$. 
The eigenvalues of the associated Jacobian matrix at $(u,w)=(\alpha,0)$ are 
\begin{equation}\label{2}
\lambda_\pm := \frac{-c\pm \sqrt{c^2-4(1-\alpha)}}{2}, 
\end{equation}
and those at $(u,w)=(1,0)$ are 
\begin{equation}\label{1}
\Lambda_\pm := \frac{-c\pm \sqrt{c^2+4(1-\alpha)}}{2}.
\end{equation}
Since $1-\alpha>0$, the point $(1,0)$ is a saddle for all $c \geq 0$. The unstable manifolds at the saddle point $(1,0)$ are tangent to the associated eigenvector $(1,\Lambda_+)$, where $\Lambda_+>0$. On the other hand, the point $(\alpha,0)$ is a stable node for $c \geq 2\sqrt{1-\alpha}$,   and a stable focus (or a stable spiral point) for $0<c<2\sqrt{1-\alpha}$. For $c=0$, the point $(\alpha,0)$ is a center.

Let us define 
\begin{equation}\label{LevelCurve}
E(u,w) := \frac{w^2}{2} - \int_{u}^1 (s-\alpha)(1-s)\,ds .
\end{equation}
In the region $u \leq 1$, the level set $E(u,w) = 0$ is a simple closed curve that contains $(u,w)=(1,0)$ on it and encloses  $(u,w)=(\alpha,0)$. See Figure \ref{Figure12}.  We also define the open set 
\begin{equation}\label{Omega}
\Omega := \{(u,w)\in \mathbb{R}^2 : E(u,w)<0, \; u<1\}.
\end{equation}
We observe that $E((3\alpha- 1)/2,0)=0$, and hence 
\begin{equation}\label{OmegaAlp}
\left\{
\begin{array}{l l}
(0,0)\in \Omega & \text{ for } 0<\alpha<\frac{1}{3}; \\ 
(0,0) \in \overline{\Omega} \text{ and } \inf_{(u,w)\in\Omega}u = 0 & \text{ for } \alpha=\frac{1}{3}; \\
(0,0) \notin \overline{\Omega} \text{ and } \inf_{(u,w)\in\Omega}u > 0 & \text{ for } \frac{1}{3}<\alpha<1.
\end{array} 
\right.
\end{equation}

\subsection{Stationary solutions}
We investigate the solutions to \eqref{KPP_Sys} with $c=0$. 
\begin{proposition}\label{Pro_stead}
Consider \eqref{KPP_Sys} with $c=0$. Then, the following hold:
\begin{enumerate}
\item There is a homoclinic orbit approaching $(u,w)=(1,0)$ as $\xi\to \pm\infty$. The trajectory lies on the simple closed loop  $E(u,w)=0$;
\item For each point $(u,w)=(u_0,0)$ with $u_0 \in (3\alpha-1/2, \alpha)$, there is a periodic orbit passing through $(u_0,0)$. The trajectory lies on the simple closed loop $E(u,w)=-E_0<0$, where $E_0:=-E(u_0,0)$.
\end{enumerate}
\end{proposition}
\begin{proof}
By the stable manifold theorem, \cite{P}, there is a solution to \eqref{KPP_Sys} with $c=0$ approaching the saddle point $(1,0)$ as $\xi \to -\infty$ in the region $\{(u,w) : u <1, w<0\}$. Furthermore, the solution satisfies $E(u(\xi),w(\xi))=0$ as long as the solution exists. Hence, $(u,w)(\xi)$ exists for all $\xi \in\mathbb{R}$, and it is bounded.

We claim that $(u,w)(\xi) \to (1,0)$ as $\xi \to +\infty$. First, we notice that the solution cannot pass through the point $(1,0)$ at some finite $\xi_0$ since it is an equilibrium point. Hence, we suppose to the contrary that there is a point $(u_0,w_0) \neq (1,0)$ on the loop $E(u,w)=0$ such that $(u,w)(\xi) \to (u_0,w_0) $ as $\xi \to +\infty$. By the local existence of the ODE \eqref{KPP_Sys}, there is small number $\tilde{\xi}>0$ such that the solution $(\tilde{u},\tilde{w})$ to \eqref{KPP_Sys} with the initial condition $(\tilde{u},\tilde{w})(0)=(u_0,w_0)$ exists on the interval $\xi \in (-\tilde{\xi},\tilde{\xi})$. Moreover, $E(\tilde{u}(\xi),\tilde{w}(\xi))=0$ holds as long as the solution exists. By the uniqueness, $(u,w)$ is extended  beyond the point $(u_0,w_0)$ along the curve $E(u,w)=0$. This is a contradiction, and we finish the proof of the first statement. 

We prove the second statement. For any $u_0 \in (3\alpha-1/2,\alpha)$, there is a positive number $E_0:=-E(u_0,0)$ such that $E(u,w)=-E_0$ is a simple closed loop that encloses $(\alpha,0)$ and lies on the region $E(u,w)<0$. By a similar argument as above, one can show the existence of a periodic orbit passing through $(u_0,0)$, since there are no equilibrium points on the curve $E(u,w)=-E_0$. We omit the details. 

On the other hand, from \eqref{KPP_Sys}, we see that 
\begin{equation}\label{period}
\xi = \int_0^\xi \frac{1}{w}\frac{du}{d\tilde{\xi}}\,d\tilde{\xi} = \int_{u(0)}^{u(\xi)}  \frac{du}{w} = \int_{u(0)}^{u(\xi)} \frac{du}{\sqrt{2}\sqrt{\int_u^1 (s-\alpha)(1-s)\,ds - E_0}}.
\end{equation}
For any $u_0 \in (3\alpha-1/2,\alpha)$, there exists a unique $u_1\in(\alpha,1)$ such that $E(u_0,0)=E(u_1,0)$ by the definition of $E$. Let 
\[
T :=  \int_{u_0}^{u_1} \frac{du}{\sqrt{2}\sqrt{\int_u^1 (s-\alpha)(1-s)\,ds - E_0}}.
\]
From \eqref{period} and symmetry, we see that $2T$ is the period of the periodic orbit passing through $(u_0,0)$. We finish the proof.
\end{proof}

\subsection{Traveling wave solutions}

We first start with the nonexistence of periodic or homoclinic orbits to \eqref{KPP_Sys} with $c\neq 0$. 
\begin{lemma}\label{Lem2}
Consider the system \eqref{KPP_Sys} with $c\neq 0$. Then, it does not admit nontrivial periodic solutions. Furthermore, it does not admits nontrivial bounded solutions satisfying $ \textstyle{\lim_{\xi \to -\infty }u =  \lim_{\xi \to +\infty }u } $ and $ \textstyle{\lim_{\xi \to \pm \infty} w = 0}$.  
\end{lemma}
\begin{proof}
Suppose that \eqref{KPP_Sys} admits a periodic solution with period $T$. Multiplying the second equation of \eqref{KPP_Sys} by $w=u'$, and then integrating the resulting equation from $0$ to $T$, we see that 
\[
-c \int_0^T w^2\,d\xi = 0,
\]
which is a contradiction since $c \neq 0$. 

The second statement can be proved in a similar fashion. We omit the details, and we finish the proof of Lemma \ref{Lem2}.
\end{proof}
 
\begin{proposition}\label{Proposi_1}
Consider the system \eqref{KPP_Sys} with $c>0$. Then, the following hold:\\
(1) Any local solution $(u,w)(\xi)$ to \eqref{KPP_Sys} passing through a point $(u_0,w_0) \in \Omega$  at $\xi=0$ can be extended over all $\xi \in [0,+\infty)$. In particular, $(u,w)\in\Omega$ for all $\xi\in[0,+\infty)$, and $(u,w)$ approaches $(\alpha,0)$ as $\xi \to +\infty$.  \\
(2) There is a unique solution to \eqref{KPP_Sys} satisfying 
\begin{equation}\label{Eq1_Thm1}
\lim_{\xi \to -\infty}(u,w)(\xi) = (1,0), \quad \lim_{\xi \to +\infty}(u,w)(\xi) = (\alpha,0).
\end{equation} 
Moreover, $(u,w)\in \Omega$ for all $\xi \in \mathbb{R}$. In particular, when $c \geq c^\ast$, $u(\xi)$ is monotonically decreasing on $\mathbb{R}$. When $0<c<c^\ast$, $u(\xi)$ decays to $\alpha$ as $\xi \to +\infty$ in an oscillatory fashion, and there is a (unique) $\xi_0$ such that $\tfrac{3\alpha-1}{2} < u(\xi_0) < \alpha$,  $w(\xi_0)=0$, and $w<0$ for all $\xi\in(-\infty,\xi_0)$.  
\end{proposition}

\begin{proof}
We first prove (1). We observe that as long as $(u,w)$ satisfies \eqref{KPP_Sys}, we have
\begin{equation}\label{Eq2_Thm1}
E(u(\xi),w(\xi))' = - cw^2 \leq  0.
\end{equation}
Hence, the solution $(u,w)$ starting at a point in $\Omega$ lies in the region $\Omega$ as long as it exists (forward in $\xi$). Moreover, the solution extends to all $\xi \in [0,+\infty)$ since $\Omega$ is bounded.  

The closure of $\Omega$ contains only two equilibrium points $(\alpha,0)$ and $(1,0)$. By Lemma \ref{Lem2} and the Poincar\'e-Bendixson theorem, \cite{P}, the $\omega$-limit set of the trajectory of \eqref{KPP_Sys} through a point in $\Omega$ is either $\{(\alpha,0)\}$ or a connected set composed of the two equilibrium points and a heteroclinic orbit connecting $(1,0)$ and $(\alpha,0)$.  
The last case is excluded since $(\alpha,0)$ is an asymptotically stable equilibrium when $c>0$; the trajectory cannot escape a neighbourhood of $(\alpha,0)$.  Hence, the $\omega$-limit set is the point $(\alpha,0)$, and the trajectory approaches $(\alpha,0)$ as $\xi \to +\infty$. The limit exists since it is an asymptotically stable equilibrium.

Now we prove (2). By \eqref{1} and the stable manifold theorem, there is a solution to \eqref{KPP_Sys} approaching the saddle point $(1,0)$ as $\xi \to -\infty$ in the region $\{(u,w): u<1, w<0\}$.  Together with statement (1), \eqref{1} implies that the trajectory is confined in the bounded region $\Omega$, and $(u,w)$ approaches $(\alpha,0)$ as $\xi \to +\infty$. In particular, it cannot leave the region $w<0$ crossing the line segment $\{(u,w): \alpha < u < 1, w=0\}$ as $w'(\xi)<0$ on it. 

For $c\in(0,c^\ast)$, we recall (see \eqref{2}) that $(\alpha,0)$ is a stable spiral point. Hence, there exists $\xi_0$ such that $w(\xi_0)=0$ and $w=u'<0$ for all $\xi \in (-\infty,\xi_0)$, and $u$ decays to $\alpha$ in an ocillatory fashion as $\xi \to +\infty$.

On the other hand, for $c \geq c^\ast$, we consider the line $w=m(u-\alpha)$ with slope $m<0$. We choose $m<0$ (for instance $m=-c/2$) such that, at any point of the line segment $\{(u,w)\in\Omega: \alpha < u, w=m(u-\alpha)\}$, it holds that 
\[
\begin{split}
\frac{dw}{du} 
& = -c + \frac{(u-\alpha)(u-1)}{w} \\
& =  -c + \frac{(u-1)}{m}  \\
& < -c - \frac{1-\alpha}{m} \leq m.
\end{split}
\]
This shows that the trajectory cannot cross the line $w=m(u-\alpha)$. See Figure~\ref{Figure13}. Hence, $w=u'<0$ hold true for all $\xi \in \mathbb{R}$. This completes the proof.
\end{proof}  

The following  can be shown in a similar fashion as Proposition \ref{Proposi_1}.
\begin{proposition}\label{Cor_Mono} 
For each $c \in [2, \infty)$, the solution $(u,w)$ to \eqref{KPP_Sys} with the initial condition $(u_0,w_0)=(0,0)$  satisfies $w>0$ and $0<u<\alpha$ for all $\xi>0$. Furthermore,  $(u,w) \to (\alpha,0)$ as $\xi \to+\infty$. 
\end{proposition} 
\begin{proof}
By investigating the direction of the vector field associated with \eqref{KPP_Sys}, we see that the trajectory starting from $(u_0,w_0)=(0,0)$ enters the first quadrant of $\mathbb{R}^2$, and it cannot leave the first quadrant crossing either the line segment $\{(u,0): 0<u<\alpha\}$ or the $w$-axis. 

We consider the line $w=m(u-\alpha)$ with $m<0$, and we see that on the line with $0  < u < \alpha$,  
\[
\begin{split}
\frac{dw}{du} 
& =  -c + \frac{(u-1)}{m}  \\
& < -c -\frac{1}{m}.
\end{split}
\]
Now we choose $m<0$ so that $-c -1/m \leq m$, that is, 
\[
\frac{-c - \sqrt{c^2-4}}{2} \leq m \leq \frac{-c + \sqrt{c^2-4}}{2}.
\]
 Hence, the trajectory (with $u<\alpha$) must be trapped in the triangular region. By applying the Poincar\'e-Bendixson theorem, we finish the proof.
\end{proof}

\begin{figure}[ht]
    \centering
    \includegraphics[width=0.7\linewidth]{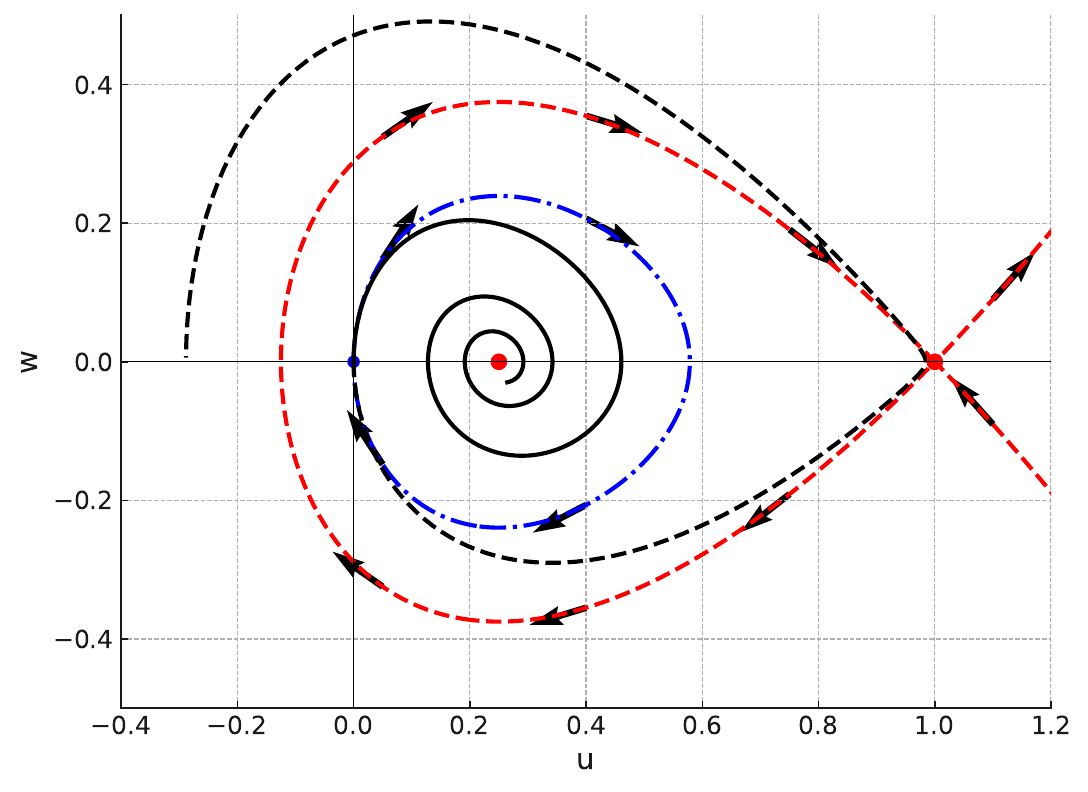}
    \caption{Phase portrait of \eqref{KPP_Sys} when $\alpha=1/4$. For $c=0$, there is a homoclinic orbit (dashed-red) connecting an unstable manifold and a stable manifold. Also, there is a periodic orbit (dashed-dotted) passing through $(0,0)$. For $c=c_\ast>0$, a trajectory starting from the unstable manifold reaches $(0,0)$ in finite time (dashed-black with $w<0$), and a trajectory passing through the origin approaches $(\alpha,0)$ in an oscillatory fashion (solid). As $c \geq 0$ increases, the solution curves $w(u)$ (dashed-black curves) shift upward. }
    \label{Figure12}
\end{figure}

\begin{figure}[ht]
    \centering
    \includegraphics[width=0.7\linewidth]{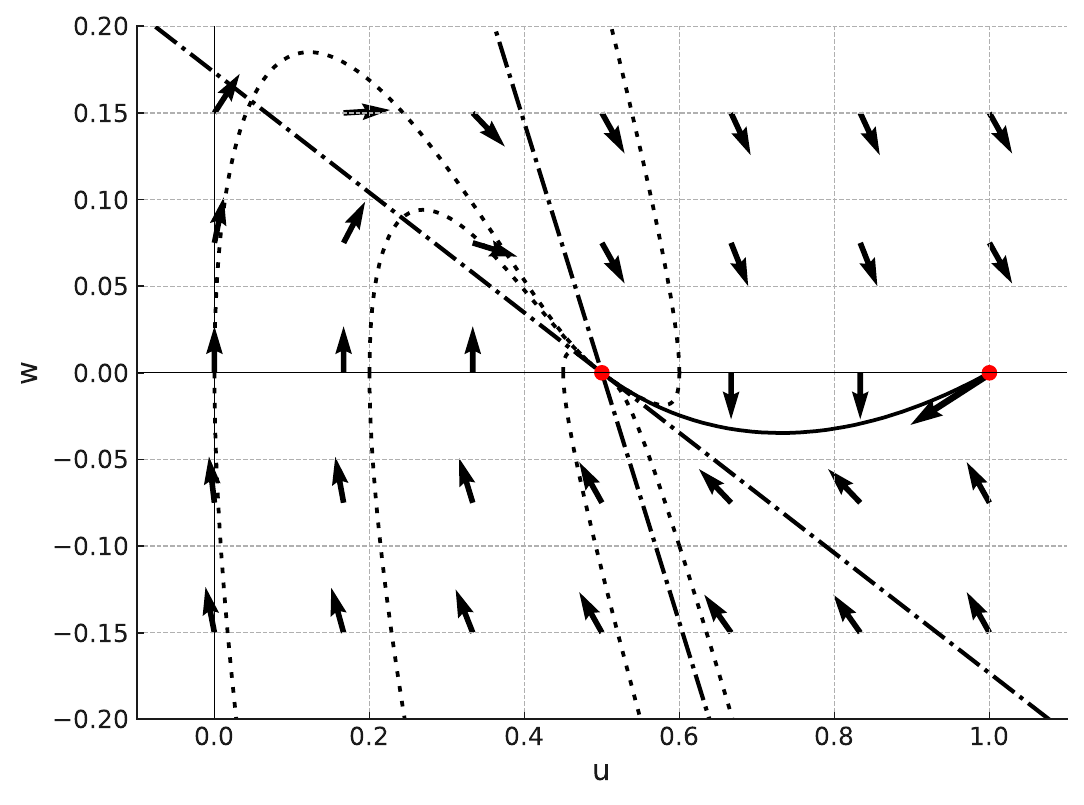}
    \caption{Phase portrait of \eqref{KPP_Sys} when $\alpha=1/2$ and $c=2\sqrt{0.8}$. For $c>2\sqrt{1-\alpha}$, the trajectory on the unstable manifold at $(1,0)$ approaches $(0.5,0)$ as $\xi \to +\infty$ without leaving the region $w<0$ (solid curve). The trajectory does not cross the dashed-dotted lines, whose slopes are $0>\lambda_+>\lambda_-$, respectively. (The lines coincide when $c=2\sqrt{1-\alpha}$.) }
    \label{Figure13}
\end{figure}

\subsection{Shooting argument}
Let $w^\pm(u;c):=w^\pm(u)$ be the solutions to the \textit{backward} initial value problems
\begin{equation}\label{BackODE1}
\left\{
\begin{array}{l l}
\frac{dw^\pm}{du}(u;c) = -c - \frac{(u-\alpha)(1-u)}{w^\pm}, \\ 
w^\pm(1;c)=0, \quad \frac{dw^\pm}{du}(1;c) = \Lambda_\mp \lessgtr 0, 
\end{array} 
\right.
\end{equation} 
 respectively, where $c \geq 0$, and $\Lambda_+$ and $\Lambda_-$ are defined in \eqref{1}. The stable manifold theorem implies that there exists a unique local solution $w^+$ (resp. $w^-$) to \eqref{BackODE1}, and it holds that $w^\pm \gtrless 0$ for all $u<1$ sufficiently close to $1$.  
By taking $\partial_c$ of  \eqref{BackODE1}, we have
\begin{equation}\label{EQ1}
\left\{
\begin{array}{l l}
\frac{d(\partial_c w^\pm)}{du}(u;c) = -1 + \frac{(u-\alpha)(1-u)}{(w^\pm)^2}\partial_c w^\pm, \\ 
\partial_c w^\pm(1;c)=0, \quad \frac{d(\partial_c w^\pm)}{du}(1;c) = \partial_c\Lambda_\mp < 0.
\end{array} 
\right.
\end{equation}  
We observe that for $c \geq 0$, we have  $\partial_c w^\pm(u)>0$ for $u<1$ sufficiently close to 1.  We omit the parameter $c$ when no confusion arises. 
\begin{lemma}\label{Lem5}
Consider \eqref{BackODE1}. Then, the following hold true: \\
(1) The maximal interval of existences $(u_c^-,1]$  of $w^-$ and $(u_c^+,1]$ of $w^+$ are finite for $c\in[0,2\sqrt{1-\alpha})$ and $c\in[0,+\infty)$, respectively. Furthermore, $u_c^\pm <\alpha$ and 
\begin{equation}\label{BackODE2}
\partial_c w^\pm(u) >0,  \quad u\in (u_c^\pm,1).
\end{equation}
(2) The mapping $c \in [0,2\sqrt{1-\alpha}) \mapsto u_c^- \in [(3\alpha-1)/2,\alpha)$ is strictly increasing and onto. \\
(3)  The mapping $c \in [0,+\infty) \mapsto u_c^+ \in (-\infty,(3\alpha-1)/2]$ is strictly decreasing and onto.  
\end{lemma} 

\begin{proof}
 We first show that  $\partial_c w^-(u) > 0$ on $u \in (u_c^-,1)$   without assuming that  $(u_c^-,1)$ is finite.  Suppose that there is $u_{\ast}\in (u_c^-,1)$ such that $\partial_c w^-(u_\ast) = 0$ and 
 \begin{equation}\label{0_0}
     \partial_c w^-(u)>0 \text{ for all } u\in(u_\ast,1).
 \end{equation}
 Then, we have   $\tfrac{d(\partial_c w^-)}{du} = -1$ at $u=u_\ast$ by \eqref{EQ1}. Since  $\tfrac{d(\partial_c w^-)}{du} <0$ at $u=1$, there exists $u_{\ast\ast}\in(u_\ast,1)$ such that $\partial_c w^-(u_{\ast\ast})=0$ by continuity. This contradicts to \eqref{0_0}. In a similar fashion, we obtain $\partial_c w^+(u) > 0$ on $u \in (u_c^+,1)$, and hence we prove \eqref{BackODE2}.

By phase plane analysis, we already proved that $\tfrac{3\alpha-1}{2}\leq u_c^-<\alpha$ for  $c\in[0,2\sqrt{1-\alpha})$ in Propositions~\ref{Pro_stead} and \ref{Proposi_1}. We note that  $w^-(u)<0$ on the maximal interval of existence $(u_c^-,1)$ and that $\textstyle\lim_{u\searrow u_c^-}w^-=0$. In particular, for $c\in[0,2\sqrt{1-\alpha})$, we have $\textstyle\lim_{u\searrow u_c^-}\tfrac{dw^-}{du}  = -\infty$.

We claim that for any given $\beta \in (0,2\sqrt{1-\alpha})$, there is a positive constant $C_\beta$, independent of $c\in[0,\beta]$, such that $\partial_c w^-(u) \geq C_\beta$ for all $u \leq \alpha$ and $c\in[0,\beta]$.  From \eqref{EQ1} and \eqref{BackODE2}, we have $\tfrac{d(\partial_c w^-)}{du} \leq -1$ for all $u\leq \alpha$. Integrating it from $u$ to $\alpha$, we have
\begin{equation}\label{Eq:0}
\begin{split}
    \partial_c w^-(u) & \geq \partial_c w^-(\alpha) + (\alpha-u)\\
    & \geq \partial_c w^-(\alpha) \\
    & \geq  \textstyle\inf_{c\in[0,\beta]}\partial_c w^-(\alpha)=:C_\beta > 0,
\end{split}
\end{equation}
where the strictly inequality is due to \eqref{BackODE2}. (We remark that $\partial_c w^-(\alpha)=0$ for all $c \geq 2\sqrt{1-\alpha}$ since $w^-(\alpha)=0$.)

We show that $u_{c_1}^- < u_{c_2}^-$ for $0\leq c_1<c_2< 2\sqrt{1-\alpha}$. Integrating \eqref{Eq:0} in $c$, we get
\begin{equation}\label{Eq:1}
w^-(u;c_2)-w^-(u;c_1) \geq C_{c_2}(c_2-c_1) >0
\end{equation}
for all $u\leq \alpha$. Here, we have chosen $\beta=c_2$. If $u_{c_1}^- \geq u_{c_2}^-$, then we must have $0=w^-(u_{c_1}^-;c_1) \geq  w^-(u_{c_1}^-;c_2)$ (recall that $w^-(u;c) < 0$ for $u\in(u_c^-,1)$ and $w^-(u;c) = 0$ for $u=u_c^-$), which contradicts to \eqref{Eq:1}.  

By continuity of $w^-$ in $c\geq 0$, we conclude that the mapping 
$$
c \in [0,2\sqrt{1-\alpha}) \mapsto u_c^- \in [(3\alpha-1)/2,\alpha)
$$
is strictly increasing and onto. This proves (2). 

We notice that \eqref{BackODE2} particularly implies that $u_c^+ \leq \tfrac{3\alpha-1}{2}$ for all $c \geq 0$, where $(\tfrac{3\alpha-1}{2},1)$ is the maximal interval of existence for $w^+(u;0)$. Since $ \textstyle\lim_{u\searrow u_0^+}\tfrac{dw^+}{du}(u;0)$ is nonzero, by continuity of $w^+$ in $c$, $w^+(u;c)$ vanishes at some point around $u_0^+=\tfrac{3\alpha-1}{2}$ for all sufficiently small $c>0$. Furthermore, at the vanishing points, we have $\tfrac{dw^+}{du} = +\infty$, which means that $u_c^+$ is finite for all sufficiently small $c>0$, (hence, for all $c>0$ by continuity). 

Now we show that $u_{c_2}^+ < u_{c_1}^+$ for $0 \leq c_1 < c_2$ following the above arguments for \eqref{Eq:0} and \eqref{Eq:1}. It is enough to show that  $C:=\textstyle\inf_{c\in[0,\infty)}\partial_c w^+(\alpha;c)>0$. If not, there is a sequence $c_k\to +\infty$ such that $\partial_c w^+(\alpha;c_k) \to 0$.  On the other hand, one can choose small $\delta>0$, uniform in $c$, such that
\begin{equation}\label{EQ2}
\frac{d(\partial_c w^+)}{du}(u;c) \leq -1/2
\end{equation}
for all $u\in(\alpha,\alpha+\delta)$ and $c\geq 0$. \eqref{EQ2} will be shown below. Integrating \eqref{EQ2} from $\alpha$ to $\alpha +\delta$, we have
\[
\partial_c w^+(\alpha+\delta;c) \leq -\frac{1}{2}\delta + \partial_c  w^+(\alpha;c).
\]
By letting $c_k \to +\infty$, we see that $\partial_c w^+(\alpha+\delta;c_k)<0$ for large $c_k$. This contradicts to \eqref{BackODE2}.

We claim that \eqref{EQ2} holds true. From \eqref{EQ1} and \eqref{BackODE2}, we have
\[
\frac{d(\partial_c w^+)}{du} \geq -1
\]
for $u\in(\alpha,1)$. Integrating it from $u$ to 1, we get $1-u \geq \partial_c w^+$ for $u\in(\alpha,1)$. Furthermore, $w^+(u;c)$ is increasing in $c$ for each $u\in (u_c^+,1)$ by (2). Hence,
\[
\frac{(u-\alpha)(1-u)}{(w^+)^2}\partial_c w^+ \leq \frac{(u-\alpha)(1-u)^2}{(w^+)^2} \leq C\frac{\delta}{(w^+)^2} \leq C'\delta < \frac{1}{2}
\]
for $u\in(\alpha,\alpha+\delta)$ and $c\in[0,\infty)$.  Combining all, together with \eqref{EQ1}, we get \eqref{EQ2}.

\end{proof}

\subsection{Discontinuous reaction: Proof of main theorems}
In this subsection, we prove Theorems \ref{thm1}--\ref{thm4}. We consider \eqref{KPP_Sys} with the step function:
\begin{equation}\label{KPP_Sys_dis}
\left\{
\begin{array}{l l}
u' = w, \\
w' = -cw  - [(u-\alpha)(1-u)]\chi_{\{u>0\}},
\end{array} 
\right.
\end{equation}
which is equivalent to  \eqref{tw}.  Observe that $(u,w)\equiv(0,0)$ satisfies \eqref{KPP_Sys_dis} in the classical sense.

We first prove Theorem \ref{thm1}.
\begin{proof}[Proof of Theorem \ref{thm1}]
As a direct consequence of Proposition \ref{Proposi_1} and Lemma \ref{Lem5}, we see that for each $0<\alpha<1/3$ (recall \eqref{OmegaAlp}), there is a unique $c=c_*=c_*(\alpha) \in (0,2\sqrt{1-\alpha})$ such that  \eqref{KPP_Sys_dis} admit the solutions $(u_\pm,w_\pm) \in C^\infty(\mathbb{R}_\pm)$ satisfying 
\begin{subequations}\label{Pr_T1}
\begin{align}
& \lim_{\xi \to -\infty}u_- =1 \quad \text{and} \quad \lim_{\xi \to 0^-}(u_-,w_-)=(0,0), \label{Pr_T11} \\
& \lim_{\xi \to 0^+}(u_+,w_+) =(0,0) \quad \text{and} \quad \lim_{\xi \to +\infty}u_+ =\alpha. \label{Pr_T12}
\end{align}
\end{subequations}
Furthermore, $u_-$ monotonically decreases on $\xi<0$, and $u_+$ approaches $\alpha$ in an oscillatory fashion as $\xi \to +\infty$ (see Figure~\ref{Figure12}).

We define the new functions $\tilde{u}_\pm(\xi)$ as follows: 
\begin{equation}\label{Pr_T2}
\tilde{u}_-(\xi):= \left\{
\begin{array}{l l}
u_-(\xi), & (\xi<0), \\
0, & (\xi \geq 0),
\end{array} 
\right.
\qquad 
\tilde{u}_+(\xi):= \left\{
\begin{array}{l l}
0, & (\xi<0), \\
u_+(\xi), & (\xi \geq 0).
\end{array} 
\right.
\end{equation}
Then, $\tilde{u}_\pm$ are in $C^1(\mathbb{R})\cap C^2(\mathbb{R}\setminus\{0\})$. Since $\lim_{\xi\to 0^-}\partial_{\xi\xi}\tilde{u}_- = \alpha$ and $\lim_{\xi\to 0^+}\partial_{\xi\xi}\tilde{u}_- = 0$, the second derivative of $\tilde{u}_-$ has a jump discontinuity at $\xi =0$. Hence, $\tilde{u}_-$  satisfies \eqref{KPP_Sys_dis} with \eqref{a1}  in the classical sense. This proves part (2) of Theorem~\ref{thm1}.

Now we prove part (1) of Theorem~\ref{thm1}. For arbitrary $L \geq 0$, we define
\begin{equation}\label{Pr_T3}
\tilde{u}_L(\xi):= \left\{
\begin{array}{l l}
u_-(\xi), & (\xi<0), \\
0, & (0 \leq \xi \leq L), \\
u_+(\xi-L), & (\xi \geq L)
\end{array} 
\right.
\end{equation}
(see Figure~\ref{NewFig1}). It is obvious that $\tilde{u}_L$ satisfies \eqref{KPP_Sys_dis} with \eqref{a2} in the classical sense. This proves statement (\textit{iii}).

Combining Proposition \ref{Proposi_1}, Lemma \ref{Lem5}, and the definition of $c_*$ above, it is straightforward to conclude that \eqref{KPP_Sys_dis} admits a strictly positive solution satisfying \eqref{a2} for all $c > c_*$. In particular, for $c_*<c<2\sqrt{1-\alpha}$, $u$ approaches $\alpha$ in an oscillatory fashion as $\xi \to +\infty$. On the other hand, for $c \geq 2\sqrt{1-\alpha}$, $u$ is monotonically decreasing. This proves statements~(\textit{i}) and (\textit{ii}). We completes the proof of part (2) of Theorem \ref{thm1}.    
\end{proof}

Next, we prove Theorem \ref{thm2}.
\begin{proof}[Proof of Theorem \ref{thm2}]
When $1/3 \leq \alpha < 1$, recalling \eqref{OmegaAlp}, we see that part (1) of Theorem~\ref{thm2} follows directly from Proposition \ref{Proposi_1}.  On the other hand, part (2) of Theorem~\ref{thm2} follows from symmetry \eqref{Symmetry} and the properties of $w^+$ in Lemma~\ref{Lem5}. We finish the proof of Theorem \ref{thm2}.
\end{proof}

We prove Theorem~\ref{thm3}.
\begin{proof}[Proof of Theorem \ref{thm3}]
Proposition \ref{Cor_Mono} implies that the set
\begin{equation*}
        A:=\{c \geq 0 : \tfrac{du(\xi;c)}{d\xi} \geq 0 \quad \textrm{for all } \xi \in\mathbb{R}\}
    \end{equation*}
    is nonempty and $2\in A$. Since it is a closed set,  $c^{\ast\ast}:=\min A$ is well-defined by continuity. Furthermore, we have $c^\ast\leq c^{**}\leq 2$ by \eqref{2}. 
    
In the case $0<\alpha\leq 1/3$, by Proposition \ref{Proposi_1}, we see that for all $c > 0=:c_{**}(\alpha)$, \eqref{KPP_Sys_dis} admits a solution satisfying \eqref{a4}. In the case $1/3<\alpha<1$, let $c_{\ast\ast}:=-c_*(\alpha)$, where $c_*$ is the speed in Theorem \ref{thm2} for which \eqref{KPP_Sys_dis} admits a solution connecting $(1,0)$ and $(0,0)$.  By Lemma \ref{Lem5} and  uniqueness of ODE, for all $c> c_{\ast\ast}$, there is a solution to \eqref{KPP_Sys_dis} connecting $(0,0)$ and $(\alpha,0)$ (recall that $(u,w)=(\alpha,0)$ is stable for $c > 0$). For all $0<c<c_{\ast\ast}$, there are no such solutions. This finishes the proof of part (1) of Theorem~\ref{thm3}. 

We prove part (2) of Theorem~\ref{thm3}. We note that $c_{\ast\ast} < c^{\ast\ast}$ by uniqueness of ODE. If $c\ge c^{**}$, the solution is monotone by the definition of $c^{**}$. If $c\in(c_{\ast\ast},c^\ast)$, the trajectory approaches $(\alpha,0)$ as $\xi \to +\infty$ in an oscillatory fashion. If $c^*<c^{**}$, we show that for $c\in[c^*,c^{**})$, there is a unique $\xi_0$  such that the local maximum of $u(\xi)$ larger than $\alpha$ is attained at $\xi=\xi_0$. Suppose that there exists such a point $\xi_0$. In other words, there exists $\delta>0$ such that
\begin{equation}\label{Eq_9}
w(\xi_0)=0, \quad w(\xi)>0 \text{ for } \xi \in (\xi_0-\delta,\xi_0), \quad w(\xi)<0  \text{ for } \xi \in (\xi_0, \xi_0+ \delta).
\end{equation}
Then, the same argument as in the proof of Proposition \ref{Proposi_1} implies that the trajectory approaches $(\alpha,0)$ without leaving the region $w<0$.  We complete the proof of Theorem~\ref{thm3}.

\end{proof}

Finally, we prove Theorem \ref{thm4}.
\begin{proof}[Proof of Theorem \ref{thm4}] 
Theorem~\ref{thm4} follows by gluing the solutions of \eqref{KPP_Sys}, constructed in Proposition~\ref{Pro_stead}, with the trivial solution  $u\equiv 0$ of \eqref{KPP_Sys_dis}, in a similar manner to the proof of Theorem~\ref{thm1}. We omit the details.
\end{proof}

\section{Further discussions}\label{Sec_3}
\subsection{Dynamics of the model}\label{Sec3_1}
Using our main results and the comparison principle established in  \cite{CKKP} (see also the Appendix of the present paper), We present results on the dynamics of solutions to the initial value problem associated with \eqref{sys}. In particular, we deduce the finite-time extinction and the finite-speed propagation with a compact support.

Let $\gamma(t)$ be the solution to the ODE
\begin{equation}\label{gamma1}
\gamma_t = (\gamma-\alpha)(1-\gamma).
\end{equation}
If $0<\gamma(0)<\alpha$, then there exists $T_\ast>0$ such that $\gamma(T_\ast)=0$.

\begin{theorem}[Finite time extinction I]
Let $\alpha\in(0,1)$ be given. For any non-negative initial data $u_0$ satisfying $\sup_{x\in\mathbb{R}}u_0 < \alpha$, there exists $T_\ast>0$ such that the solution to the initial value problem \eqref{sys} satisfies $u(x,t) \equiv 0$ for all $x\in\mathbb{R}$ and $t\geq T_\ast$.  
\end{theorem}
\begin{proof}
The solution $\gamma(t)$ to the ODE \eqref{gamma1} with the initial value $\gamma(0)=\sup_{x\in\mathbb{R}}u_0<\alpha$ is a super-solution of \eqref{sys}, and $u\equiv 0$ is a sub-solution of \eqref{sys}. Since  $0\leq u_0<\alpha$, the conclusion follows from the comparison principle.  
\end{proof}

Let $\bar{u}$ be the compactly supported nontrivial solution to \eqref{tw} with $c=0$ in Theorem \ref{thm4}.

\begin{theorem}[Finite time extinction II]
Let $\alpha\in(0,1/3)$ be given. For any nonnegative initial data $u_0$ satisfying  $\inf_{\xi\in\mathbb{R}}(\bar{u}-u_0) >0$, there exists $T_\ast>0$ such that the solution to the initial value problem \eqref{sys} satisfies $u(x,t) \equiv 0$ for all $x\in\mathbb{R}$ and $t\geq T_\ast$.  
\end{theorem}
\begin{proof}
Since the infimum is strictly positive, we can choose sufficiently small $c>0$  such that \eqref{tw} admits a solution $\bar{u}_c$ satisfying $0\leq u_0 \leq \bar{u}_c$ for all $\xi \in\mathbb{R}$ and \eqref{a4} (see part (1) of Theorem~\ref{thm3} and Figure~\ref{NewFig3}). By comparison principle and symmetry \eqref{Symmetry}, the conclusion follows.  
\end{proof}
\begin{theorem}[Finite time extinction III and Free-boundaries]
Let $\alpha\in(0,1)$ be given. For any nonnegative initial data $u_0\in C_c^\infty(\mathbb{R})$ satisfying $\max_{x\in\mathbb{R}}u_0 <1$, the solution to \eqref{sys} is compactly supported for all $t\geq 0$. In particular, if $\alpha>1/3$, the solution becomes identically zero in finite time. 
\end{theorem}
\begin{proof}
In the case $0<\alpha<1/3$, we may choose $x_1$ such that $u_0(x) \leq \bar{u}_1(x-x_1)$ at $t=0$, where $\bar{u}_1$ is the solution to \eqref{tw} with $c_\ast>0$ connecting $1$ and $0$ (see part (2) of Theorem~\ref{thm1}). The result follows from the comparison principle and symmetry \eqref{Symmetry}.

For the case $1/3 \leq  \alpha <1$, we make use of the solution connecting $1$ and $0$ (see part (2) of Theorem~\ref{thm2}). The results follow from comparison and symmetry in a similar fashion. We finish the proof.
\end{proof}

\subsection{Stability of traveling wave and stationary solutions}\label{Sec3_2}

In this study, the traveling waves connecting $1$ and $\alpha$ that do not attain the value $0$ exhibit characteristics analogous to those of Fisher–KPP type waves (see Figure~\ref{NewFig1}). Therefore, their qualitative behavior can be readily anticipated. Among them, the monotone solutions connecting $1$ and $\alpha$ described in Theorem~\ref{thm1} are spectrally unstable in the standard $L^2(\mathbb{R})$ space, but are spectrally stable in exponentially weighted $L^2(\mathbb{R})$ spaces \cite{Sat}. In contrast, the oscillatory solutions in Theorem~\ref{thm2} are absolutely unstable. For further details, we refer the reader to \cite{Sand} and the references therein. On the other hand, when the traveling wave solution attains the value $0$, it gives rise to a free boundary that separates the region of positive values from the zero region.

Numerical simulations (see Figures~\ref{Fig5_1}--\ref{Fig8_1}) suggest that the monotone traveling wave (or stationary) solutions of \eqref{sys} with free boundaries (i.e., the monotone solutions of~\eqref{tw} satisfying \eqref{a1} or \eqref{a4}) are stable.  
Nevertheless, the discontinuity in the reaction term introduces substantial challenges in rigorously proving the stability of such waves. First, since the reaction term is discontinuous at $u=0$, linearization near $u=0$ is not feasible. This prevents the direct application of spectral analysis techniques such as those used in~\cite{LRP}. Furthermore, the local behaviors of the reaction term $f$ near the stable equilibriums 0 and 1 are qualitatively different. More precisely, solutions to \eqref{Model} approach 0 in finite time, whereas they approach 1 only asymptotically as $t\to +\infty$. This complicates the construction of sub- and super-solutions as done in~\cite{Giletti2023,Ho,XJMY2}, where the stability of monotone traveling waves with free boundaries has been studied, using the arguments of \cite{Fi,FM}.

\begin{figure}[ht]
\begin{tabular}[]{@{}c@{}@{}c@{}} 
\resizebox{0.5\textwidth}{!}{\includegraphics{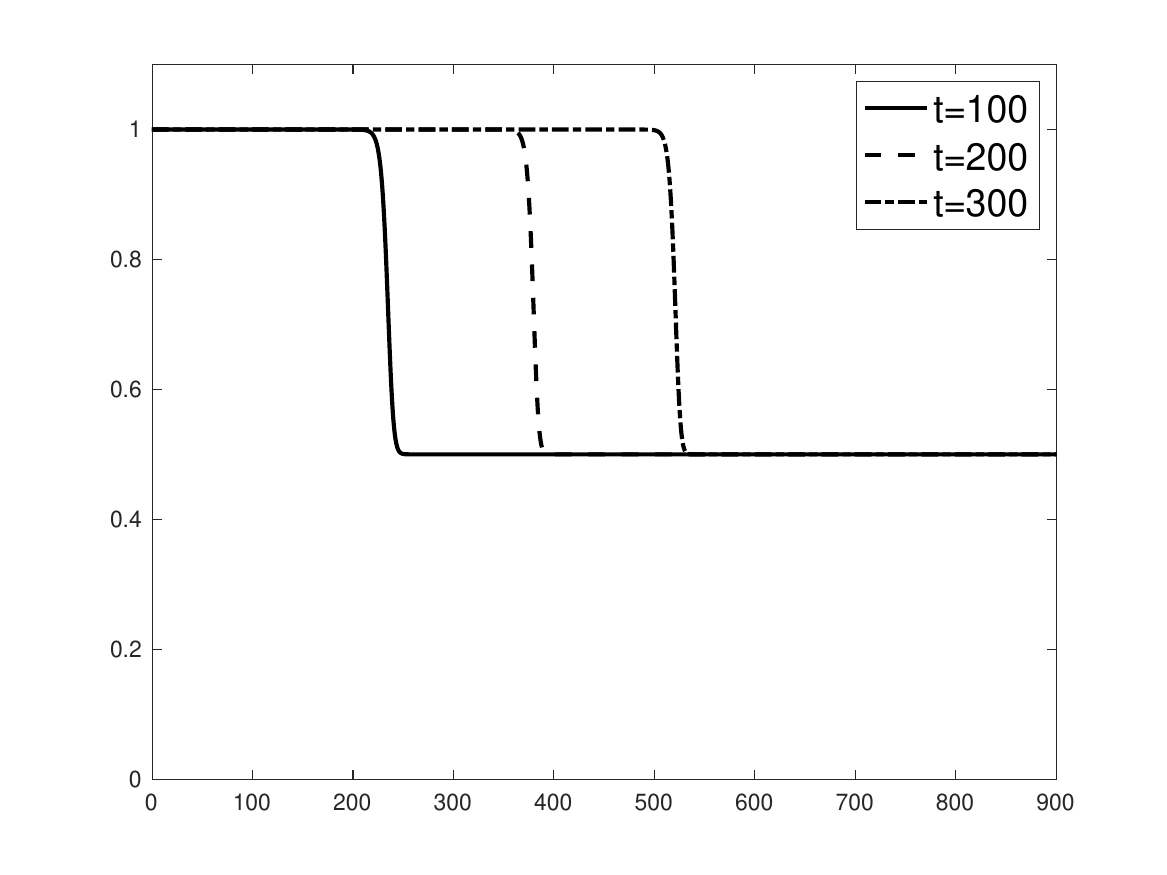}}  & \resizebox{0.5\textwidth}{!}{\includegraphics{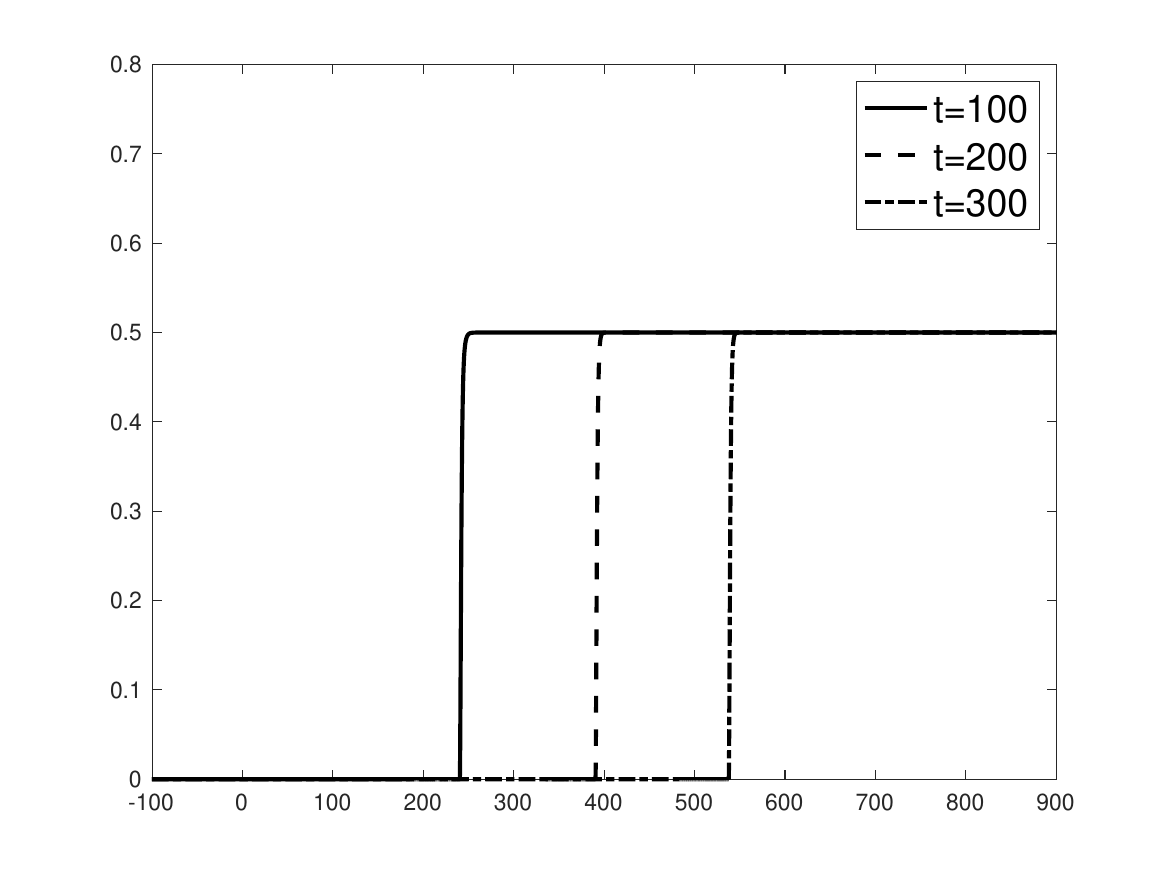}} \\
(a) $u_0(x)=\tfrac{1-\alpha}{2}(\tanh(-0.1x)+1)+\alpha$ & (b) $u_0(x)=\tfrac{\alpha}{2}(\tanh(0.1x)+1)$
\end{tabular}
\caption{The numerical solutions of \eqref{sys} with $\alpha=0.5$ for the initial data $u_0$. In (a), the solution asymptotically converges to a monotone traveling solution with speed $\approx \sqrt{2}$ corresponding to the minimal speed $c^\ast$.  In (b), the solution asymptotically converges to a monotone traveling wave solution with speed $\approx 1.472$. This suggests that the minimal speed of the monotone traveling waves connecting 0 and $\alpha=0.5$ is larger than $c^\ast = \sqrt{2}$ and smaller than 2 (see Theorem \ref{thm3}). }
\label{Fig5_1}
\end{figure}

\begin{figure}[ht]
\begin{tabular}[]{@{}c@{}@{}c@{}} 
\resizebox{0.5\textwidth}{!}{\includegraphics{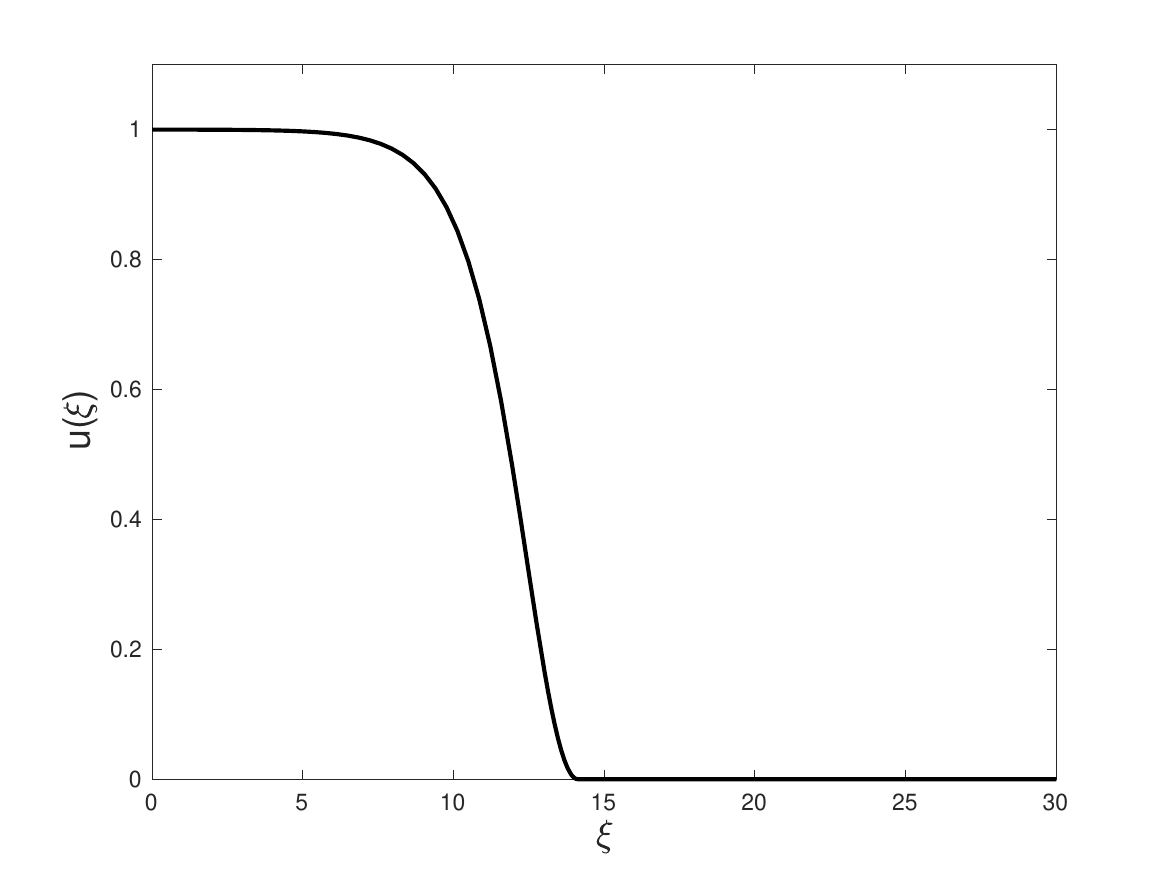}}  & \resizebox{0.5\textwidth}{!}{\includegraphics{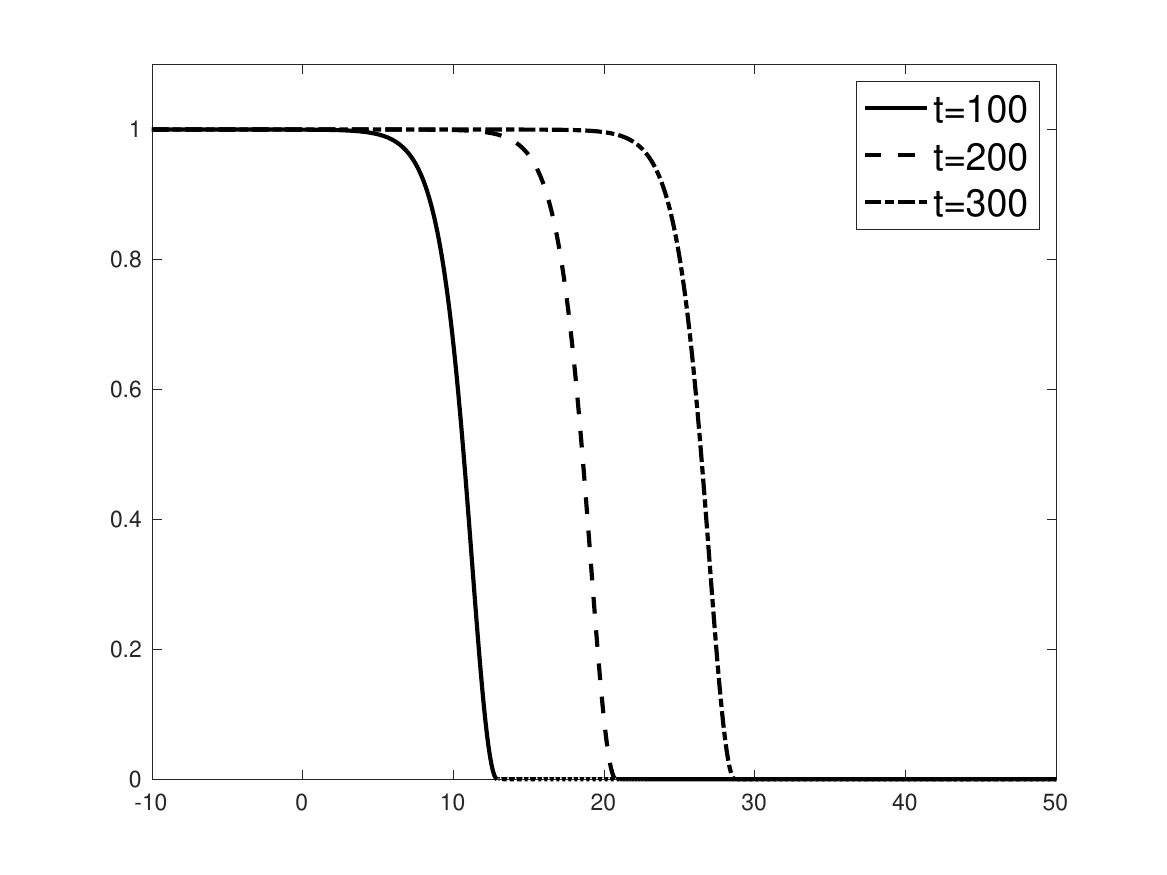}} \\
(a)  & (b) 
\end{tabular}
\caption{(a): Profile of the solution to \eqref{tw}  connecting 1 and 0 (Theorem~\ref{thm1}). For $\alpha=0.3$,  $c_*(\alpha) \approx 0.0792$.  (b): The numerical solution to \eqref{sys} with  $\alpha=0.3$ for the initial data $u_0(x)=\tfrac{1}{2}(\tanh(-0.1x)+1)$. The propagation speed of the front is approximately $0.08$.}  
\label{Fig1_1}
\end{figure}

\begin{figure}[ht]
\begin{tabular}[]{@{}c@{}@{}c@{}} 
\resizebox{0.5\textwidth}{!}{\includegraphics{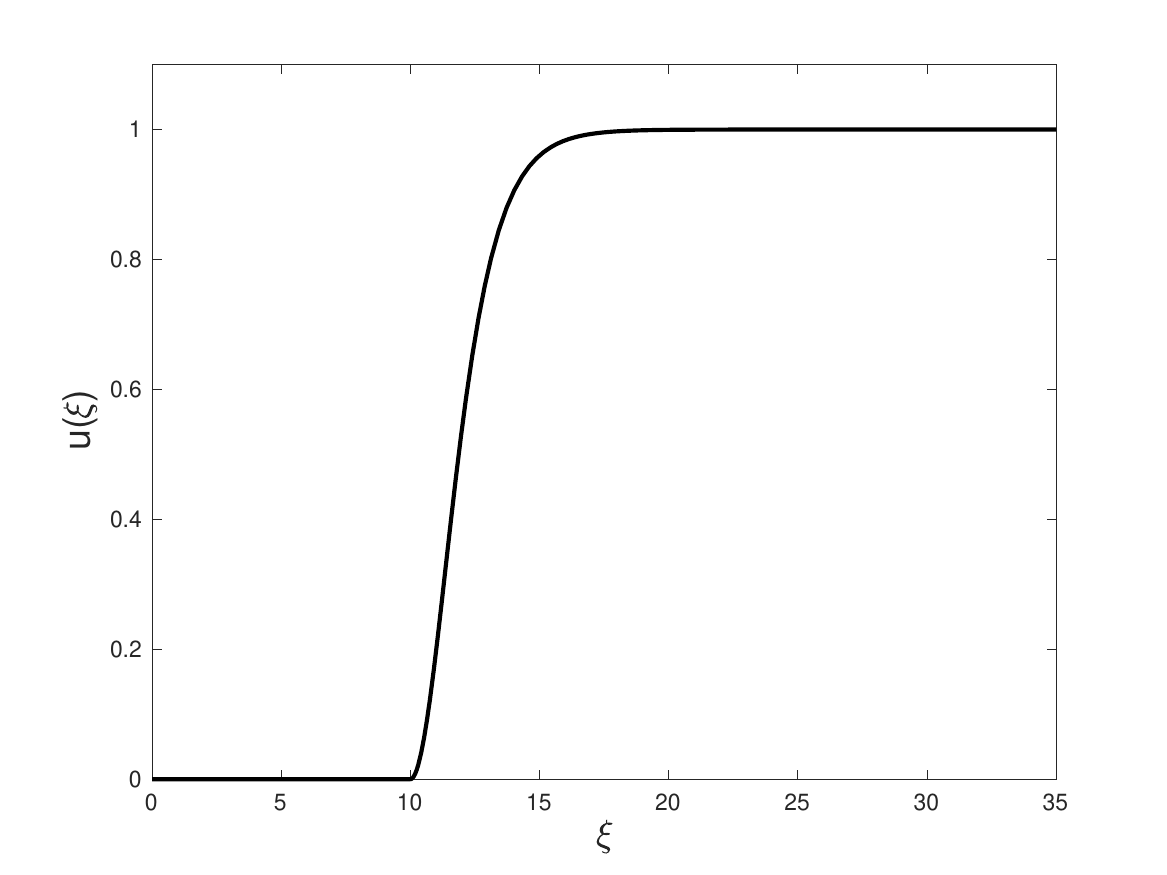}}  & \resizebox{0.5\textwidth}{!}{\includegraphics{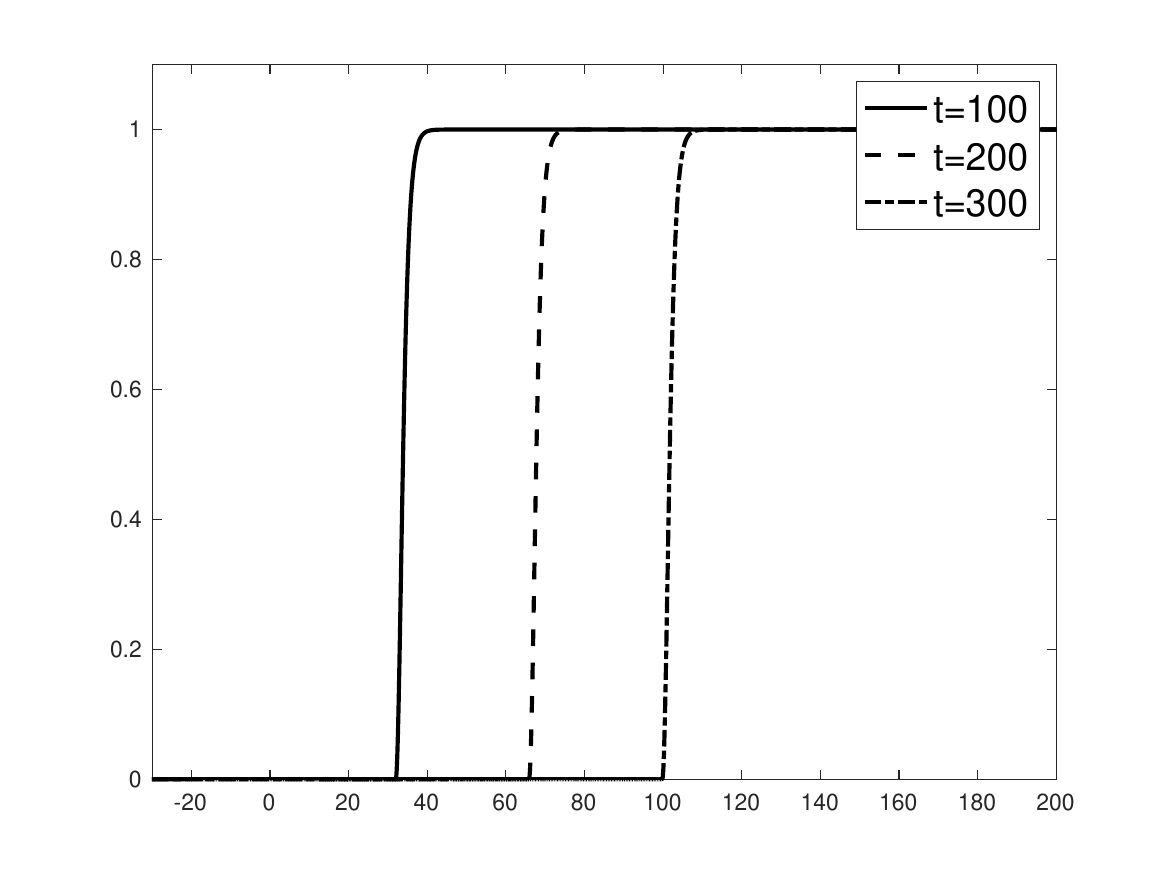}} \\
(a)  & (b) 
\end{tabular}
\caption{(a): Profile of the solution to \eqref{tw} connecting 0 and 1 (Theorem~\ref{thm3}). For $\alpha=0.5$,  $c_{\ast\ast}(\alpha) \approx 0.339$.  (b): The numerical solution to \eqref{sys} with  $\alpha=0.5$ for the initial data $u_0(x)=\frac{1}{2}(\tanh(0.1x)+1)$. The propagation speed of the front is approximately $0.336$.  }
\label{Fig3_1}
\end{figure}

\begin{figure}[ht]
\begin{tabular}[]{@{}c@{}@{}c@{}} 
\resizebox{0.5\textwidth}{!}{\includegraphics{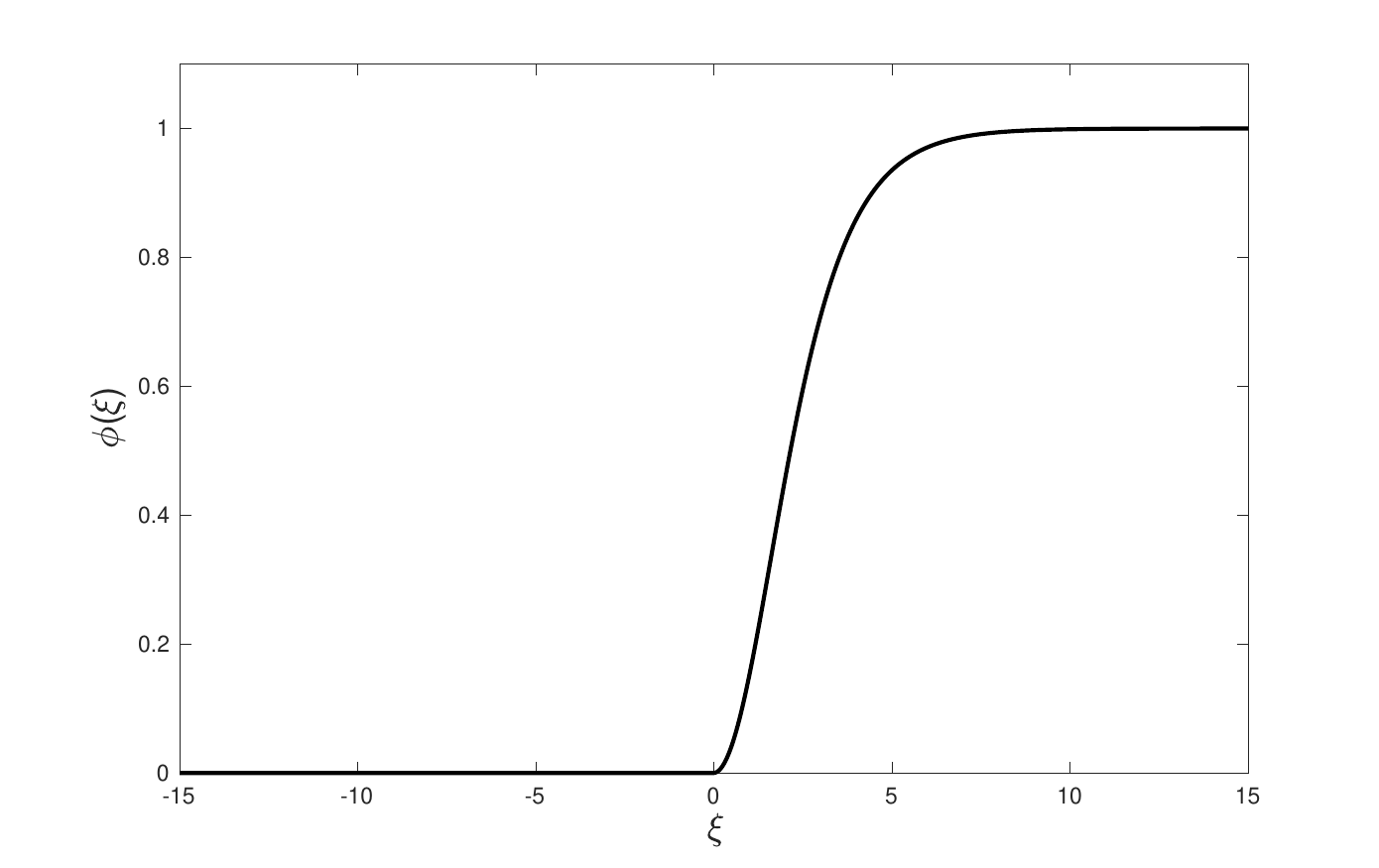}}  & \resizebox{0.5\textwidth}{!}{\includegraphics{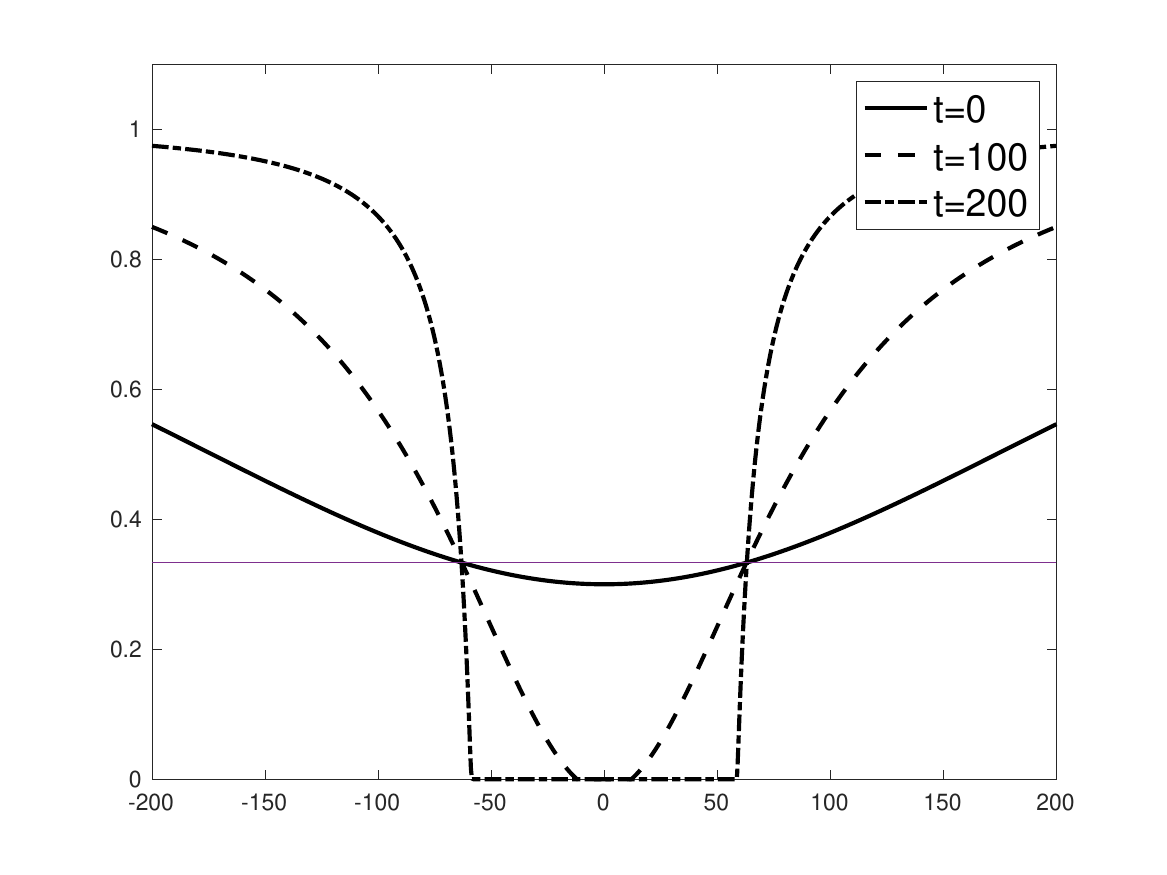}} \\
(a)  & (b) $u_0(x)=1-0.7\operatorname{sech}(0.005x)$ \\
\resizebox{0.5\textwidth}{!}{\includegraphics{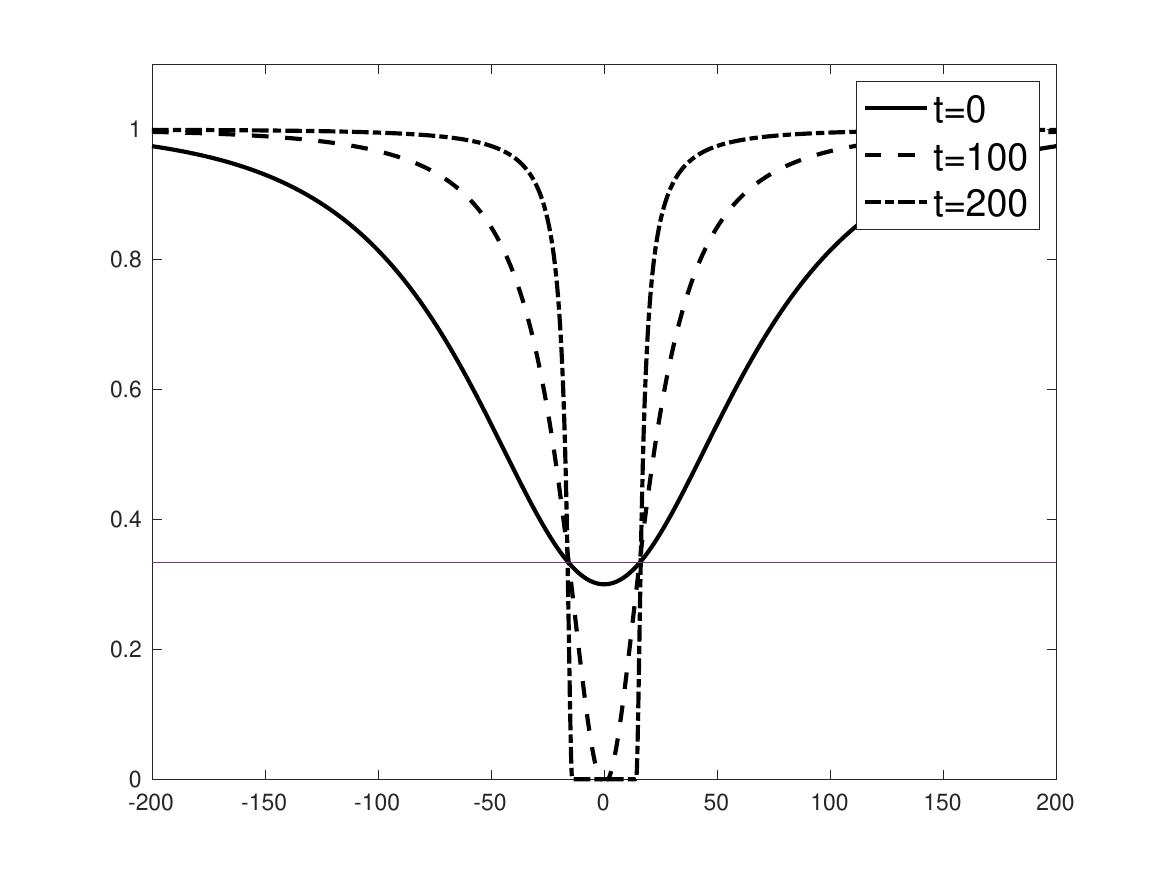}}  & \resizebox{0.5\textwidth}{!}{\includegraphics{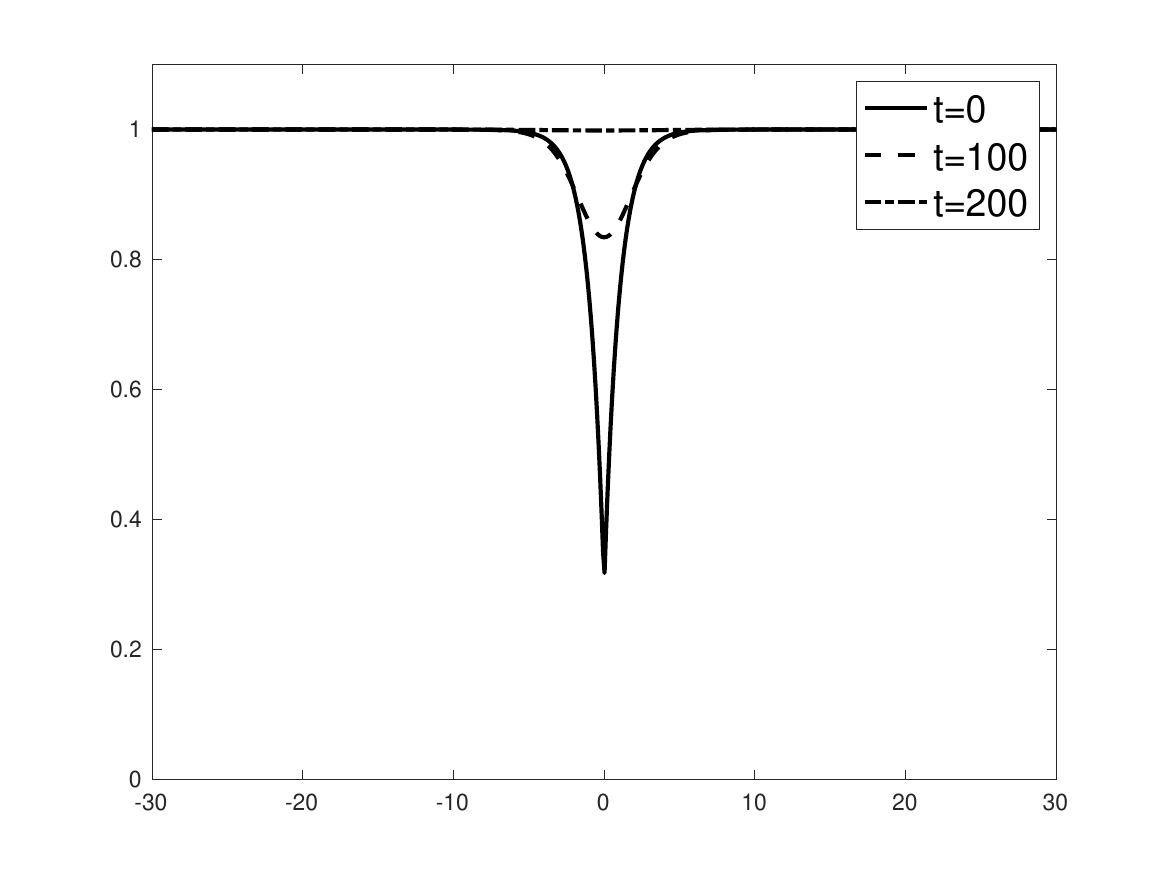}} \\
(c) $u_0(x)=1-0.7\operatorname{sech}(0.02x)$ & (d) $u_0(x)=1-0.7e^{-2|x|}$
\end{tabular} 
\caption{Numerical solutions of \eqref{sys} with $\alpha=1/3$ are presented. A stationary solution (Theorem~\ref{thm4}.(2)) is shown in (a). In (b)-(d), the numerical solutions of \eqref{sys} for different initial data $u_0$ are presented. (The horizontal line represents $u=1/3$.) In (b) and (c), the solutions converge to combinations of the stationary solution (see Theorem~\ref{thm4}.(2)), respectively. In contrast, in (d), the solution converges to 1, even though some part of the initial data lies below the unstable equilibrium at $u = 1/3$. In this case, the sharp gradient induces a stronger diffusive effect than the (unstable) reaction effect near its minimum point.}
\label{Fig8_1}
\end{figure}

\begin{figure}[ht] 
\begin{center}
        \includegraphics[width=0.9\linewidth]{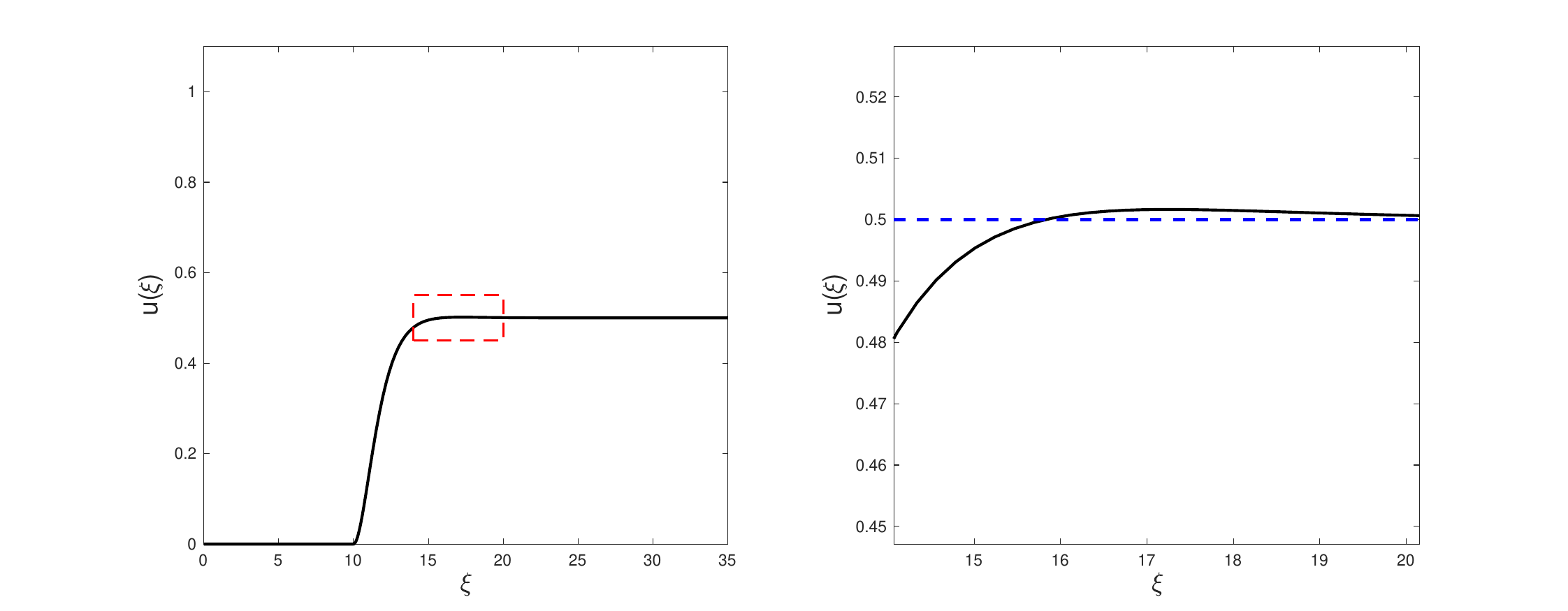}
\end{center}
\caption{A pushed wave related to Theorem~\ref{thm3} when $\alpha=0.5$ and $c=1.4145>c^\ast=\sqrt{2}$. This plot indicates $c^*<c^{**}$ numerically. }  
\label{Fig4_1}
\end{figure}

\section{Appendix}

\subsection{Global existence and comparison principle}
In \cite{CKKP}, the global existence of \textit{weak} solutions to \eqref{sys} with $0 < \alpha < 1/3$ is established, as well as the comparison principle under certain assumptions. In fact, we can deduce the same properties for \eqref{sys} with $0 < \alpha < 1$ by following the proofs given in \cite{CKKP} in a parallel manner.

Consider  
\begin{equation}\label{sys_f}
\left\{
\begin{array}{l l}
 u_t - u_{xx} = f(u), \quad t> 0, \; x\in \mathbb{R}, \\ 
u(x,0)=u_0(x),
\end{array} 
\right.
\end{equation}
where $f$ is a function satisfying  
\begin{equation}\label{Assu1}\tag{A1}
f \in \text{Lip}_{\text{loc}}(\mathbb{R}^-)\cap\text{Lip}_{\text{loc}}(\mathbb{R}^+), \quad \lim_{s\to 0^-}f(s) = f^* \geq 0, \quad \lim_{s\to 0^+}f(s) = -f_*<0.    
\end{equation}
We further assume that there exist positive constants $a^*$ and $h$ such that  
\begin{equation}\label{Assu2}\tag{A2}
   a^*<h, \quad  f(s)<0 \text{ on } (0,a^*), \quad f(s)>0 \text{ on } (a^*,h], \text{ and }\int_0^h f(s)\,ds = 0.
\end{equation}

We note that for $f(u)=(u-\alpha)(1-u)\chi_{\{u > 0\}}$, \eqref{Assu2} holds if and only if $0<\alpha<1/3$. We refer to Section 5.1 of \cite{CKKP} for the notions of \textit{weak} solutions, as well as \textit{weak super- and subsolutions} to \eqref{sys_f} satisfying \eqref{Assu1}.
\begin{theorem}[\cite{CKKP}, Section 5.2--5.3]\label{Thm_CKKP}
Suppose that $f$ satisfies \eqref{Assu1}--\eqref{Assu2} and there exists constants $b_0,b_1 > 0$ such that 
\begin{equation}\label{Assu3}
f(u) \leq b_1 u + b_0 \quad \text{for } u > 0.    
\end{equation}
Then, for a nonnegative and bounded function $u_0\in C_{\mathrm{loc}}^{1+\beta}(\mathbb{R})$, there is a weak solution $u$ of \eqref{sys_f} globally in time and $u \in C_{\mathrm{loc}}^{1+\beta,(1+\beta)/2}(\mathbb{R}\times [0,+\infty))$.
Furthermore, the following holds true:
\begin{enumerate}[(i)]
    \item \emph{(Comparison)} Let $u_1(x,t)$ and $u_2(x,t)$ be super- and subsolutions of \eqref{sys_f}, respectively, satisfying $u_1(x,0) \leq u_2(x,0)$. If $\partial_t u_1,\partial_t u_2 \in L^2(Q)$ for any compact subset $Q \Subset \mathbb{R}\times (0,\infty)$, then $u_1(x,t) \leq u_2(x,t)$ for all $(x,t) \in \mathbb{R}\times (0,\infty)$.
    \item \emph{(Uniqueness)} The weak solution of \eqref{sys_f} is unique among functions such that $\partial_t u \in L^2(Q)$ for any compact subset $Q \Subset \mathbb{R}\times (0,\infty)$.
    \item \emph{(Non-negativity)} If $u$ is a weak solution with $u(x,0) \geq 0$, then $u(x,t) \geq 0$ for all $(x,t) \in \mathbb{R}\times (0,\infty)$.
\end{enumerate}
\end{theorem} 

Following \cite{CKKP}, it is straightforward to check that \eqref{Assu2} is unnecessary and the following statement also holds.
\begin{theorem}
Suppose that $f$ satisfies \eqref{Assu1} and \eqref{Assu3}. Then, the same assertions as in Theorem \ref{Thm_CKKP} hold true.
\end{theorem}

\section*{Acknowledgments}
This work was supported by the National Research Foundation of Korea grant funded by the Ministry of Science and ICT, under contract numbers RS-2024-00340022 (WC), NRF-2022R1C1C2005658 (JB), and RS-2024-00347311 (YK). JB was partially supported by the ANR Project HEAD ANR-24-CE40-3260.

\subsection*{Data availability statement}
All data that support the findings of this study are included within the article (and any supplementary files).

\subsection*{Conflict of interest}
The authors declare that they have no conflict of interest.
 
\bibliographystyle{plain}
\bibliography{references}

 \end{document}